\documentclass{amsart}
\usepackage{mathrsfs}
\usepackage{amssymb}
\usepackage{amsmath}
\usepackage{amsfonts}
\usepackage{tikz}
\usepackage{amsmath,bm}
\usepackage{comment}
\usepackage{enumerate}
\usepackage{mathtools}
\usepackage[all]{xy}

\usepackage{bm}
\usepackage{multirow}
\input diagxy
\xyoption{curve}

\newtheorem{theorem}{Theorem}[section]
\theoremstyle{plain}

\newtheorem{note}[theorem]{Note}

\newtheorem{corollary}[theorem]{Corollary}

\newtheorem{definition}[theorem]{Definition}
\newtheorem{example}[theorem]{Example}

\newtheorem{lemma}[theorem]{Lemma}

\newtheorem{proposition}[theorem]{Proposition}
\newtheorem{remark}[theorem]{Remark}

\numberwithin{equation}{section}
\newcommand{\A}{\mathcal{A}}

\newcommand{\QQ}{\mathscr{Q}}
\newcommand{\PP}{\mathcal{L}}

\newcommand{\p}{\mathcal{L}}
\newcommand{\jj}{\mathsf{j}_{\A}}
\newcommand{\jbj}{\mathsf{j}_{\B}}
\newcommand{\kk}{\mathscr{K}}

\newcommand{\B}{\mathcal{B}}
\newcommand{\C}{\mathcal{C}}
\newcommand{\D}{\mathcal{D}}

\newcommand{\EL}{\mathcal{E}_{\leqs}}

\newcommand{\V}{\mathcal{V}}

\newcommand{\N}{\mathbb{N}}
\newcommand{\U}{\mathcal{U}}

\newcommand{\IDA}{Id_{\mathcal{A}}}
\newcommand{\IDB}{Id_{\mathcal{B}}}

\newcommand{\IDD}{Id_{\mathcal{D}}}

\newcommand{\Z}{\mathbb{Z}}

\newcommand{\cl}{\mathsf{cl}}
\newcommand{\ov}{\overline}

\newcommand{\da}{\mathord\downarrow}
\newcommand{\leqs}{\leqslant}
\newcommand{\POSGRL}{{\mathsf{PoSgr}_\leq}}
\newcommand{\MPOSGR}{{\mathsf{PoSgr}_s}}
\newcommand{\MPOSet}{{\mathsf{PoS}_s}}
\newcommand{\POSet}{{\mathsf{PoS}}}

\newcommand{\CPS}{{\mathsf{PoSgr^-}}}
\newcommand{\POSG}{{\mathsf{PoSgr}}}
\newcommand{\SET}{{\mathsf{Set}}}
\newcommand{\QUAN}{{\mathsf{Quant}}}
\newcommand{\UQUAN}{{\mathsf{UQuant}}}
\newcommand{\AQUAN}{{\mathcal A}{\mathsf{Quant}}}
\newcommand{\MSEM}{ {\mathsf{MSl_s}}}
\newcommand{\MQUAN}{ {\mathsf{Quant_s}}}
\newcommand{\UMQUAN}{ {\mathsf{UQuant_s}}}
\newcommand{\JQUAN}{{\mathcal D}{\mathsf{Quant}}}
\newcommand{\DQUAN}{{\mathcal D}{\mathsf{Quant}}}
\newcommand{\wid}{\widehat}

\begin{document}
	\title{Reflectors to quantales}
	
	\author[X. Zhang]{Xia Zhang}
	
	\address{\textbf{Xia Zhang}\\
		School of Mathematical Sciences, South China Normal University, 510631 Guangzhou, China}
	\email{xzhang@m.scnu.edu.cn}
	
	\author[J. Paseka]{Jan Paseka$^{*}$}
	
	\address{\textbf{Jan Paseka}, Department of Mathematics and Statistics, Faculty of Science, Masaryk University, Kotl\'{a}\v{r}sk\'{a} 2, CZ-611 37 Brno, Czech Republic}
	\email{paseka@math.muni.cz}

	\author[J.J. Feng]{Jianjun Feng}
	\address{\textbf{Jianjun Feng}\\
		School of Mathematical Sciences, South China Normal University, 510631 Guangzhou, China}
	\email{920782067@qq.com}

	\author[Y.D. Chen]{Yudong Chen}
	\address{\textbf{Yudong Chen}\\
		School of Mathematical Sciences, South China Normal University, 510631 Guangzhou, China}
	\email{yudongchen@m.scnu.edu.cn}

	\begin{abstract} In this paper, we show that marked quantales have a reflection into quantales. To obtain the reflection we construct free quantales over marked quantales using appropriate
		lower sets.
		A marked quantale is a posemigroup in which certain admissible subsets are required to have joins, and multiplication distributes over these. Sometimes are the admissible subsets in question specified by means of a so-called selection function.  A distinguishing
		feature of the study of marked quantales is that a small collection of
		axioms of an elementary nature allows one to do much that is traditional at the
		level of quantales. The axioms are sufficiently general to include as
		examples of marked quantales the classes of  posemigroups, $\sigma$-quantales,
		prequantales and quantales. Furthermore, we
		discuss another reflection to quantales obtained by the injective hull of a posemigroup.
		
		
	\end{abstract}
	
	\thanks{$^*$ Corresponding author}
	\maketitle
	
	\section*{Introduction}

This paper is a continuation of our study on injectivity for several classes of ordered structures (see \cite{ZL,ZL2,ZL3,ZP1}) based on quantale-like structures.
The study of such structures dates to the 1930s when
M. Ward and R. P. Dilworth \cite{ward} started their research on residuated
structures, motivated by ring theoretic considerations.

The term ``quantale" was suggested by C.J.~Mulvey at the Oberwolfach Category Meeting \ (see \cite{mulvey}) 
as an attempt to provide a possible setting for constructive foundations of quantum mechanics, as well as to study the spectra of non-commutative
C*-algebras, which are locales in the commutative case. A crucial moment in the development of the theory of quantales was the realization that quantales give  semantics for propositional linear logic in the same way that Boolean algebras provide  semantics for classical propositional logic (see \cite{Y}).

Quantales arise naturally as lattices of ideals, subgroups, or other suitable substructures of algebras. Quantales have many significant applications. As a valued domain,  quantales have been applied to the study of enriched categories (\cite{LZ}), logic (\cite{Y}), logic  algebras (\cite{Resende, RY}), fuzzy algebras, fuzzy topologies, and fuzzy domains (\cite{H}, \cite{W}). There are numerous algebraic and categorical aspects of quantales as well (\cite{CW, KP, R, Stubbe}). Here, we mainly recall their utilization of injectivities and completions.
	
As far as we know, the earliest work on injectivity investigations where quantale concepts were involved  was due to Bruns and Lakser (\cite{Bruns}, 1970) and independently
Horn and Kimura (\cite{horn}, 1971), where injective hulls of semilattices were constructed. In both works, the injective semilattices were characterized by completeness and a distributivity property, precisely frames as one calls afterward.	A similar manner was soon applied into $S$-systems over a semilattice by Johnson, Jr., and McMorris (\cite{JJM}, 1972), where complete lattices in which arbitrary joins are $S$-distributive constitute injective $S$-systems. {The terminologies admissible subsets and $\A$-ideals of a posemigroup defined in the present manuscript are initially derived from the above papers.} Surprisingly, it was only in 2012 that Lambek, Barr, Kennison, and  Raphael (\cite{LB}) introduced the term ``quantale" to the study of injective hulls for pomonoids, firstly came into people's sight for injectivivity research. Later on, Zhang and Laan successively generalized the results of injective pomonoids first to the posemigroup case (\cite{ZL}, 2014) and later to $S$-posets (\cite{ZL2}, 2015) and  ordered $\Omega$-algebras (\cite{ZL3}, 2016). Another approach to injectivity can be found in \cite{ZP1}, by Zhang and Paseka, for $S$-semigroups, where the generalizations in the framework of semicategories were correspondingly studied. In these works, quantale-like structures play an essential  role, named $S$-quantales, sup-algebras, and $S$-semigroup quantales, respectively.

Besides standard (pure) ordered algebras, the implicational part of quantales is formalized by some logic algebras, say, quantum B-algebras (\cite{R1}, \cite{RY}).
The completion of a quantum B-algebra $X$ was constructed so that a quantale can be viewed as an injective hull of $X$ (\cite{R2}). Some concepts for quantales survive in the framework of quantum B-algebras.

Other sources of our study are marked quantales \cite{paseka}	and partial frames \cite{pasekaf}. A small collection of axioms of an elementary nature allows one  to produce a tractable theory and
do much that is traditional at the level of quantales and frames. Examples of marked  quantales and partial frames include 	posemigroups and meet-semilattices,	m-semilattices and distributive lattices, $\sigma$-quantales and
$\sigma$-frames, prequantales and preframes, and quantales  and frames, respectively. Some of these classes apply in the theory of computing  to present programming as a mathematical activity and reasoning on programs as a form of mathematical	logic \cite{scott}.

A marked quantale is a posemigroup in which specific admissible subsets from a given
set $\A$ of subsets of our semigroup  must have joins, and multiplication distributes over these. The admissible  subsets in question are sometimes specified using a so-called selection function.
	
Inspired by the results mentioned above, this paper is determined to find reflectors for the category of quantales from which natural completions arise. Reflectors for the category of sup-lattices have been considered in \cite{KPV}, where cut-continuous pomonoids were investigated. 	The primary	purpose of this manuscript is to develop the theory in one place to its full extent. We obtain two kinds of reflectors in this work, whose reflective 	subcategories are entirely  relied on $\A$-ideals
which are lower sets closed under admissible joins and closed subsets, correspondingly.

It has been observed that $\QQ(S)$, the set of  closed subsets of a given posemigroup $S$, is precisely the $\EL$-injective hull for $S$ (see \cite{ZL} and \cite{XHZ}). We will show in the current work that $\QQ(S)$ indeed carries out a $\QUAN$-reflection (see also \cite{XZH}) and provides an order-embedding preserving functor. Simultaneously, the quantale $\IDA(S)$, the collection of all $\A$-ideals of $S$, contributes to another reflection. It is a free quantale over a marked quantale.	
	
The necessary categorical background can be found in \cite{Joy of cats}; we recommend \cite{EGHK,KP,R} for algebraic notions. For more information on the alternative formalisms and axioms used in the
literature concerning partial frames and  partial quantales, we refer the reader to Section 11 of \cite{frith}. Although we tried to make the paper as self-contained as possible, it is expected from the reader to be acquainted with basic concepts of category theory.

    \section{Preliminaries}

	\subsection{Posemigroups and ${\mathcal A}$-quantales}
	
	In this work, we will pay attention to the category of posemigroups and its subcategory $\QUAN$ of quantales, and some more subcategories between them. As usual, a {\bf posemigroup} $(S,\cdot,\leqs)$ is a semigroup
	$(S,\cdot)$ equipped with a partial ordering $\leqs$ which is
	compatible with the semigroup multiplication, that is, $a_1 a_2\leqs
	b_1 b_2$ whenever $a_1\leqs b_1$ and $a_2\leqs b_2,$ for any
	$a_1,a_2,b_1,b_2\in S$. For posemigroups $S,T$, a monotone mapping $f:S\to T$ which fulfills
	\begin{equation}\label{eq-submul}
		f(a)  f(b)\leqs f(a b),
	\end{equation}
	for all $a,b\in S$, is said to be {\bf submultiplicative} (\cite{LB}, \cite{ZL}). If the inequality in (\ref{eq-submul}) turns out to be an equality, then $f$ is a {\bf posemigroup morphism}. The category of posemigroups with posemigroup morphisms is denoted by $\POSG$.
	
	To simplify the formulation of basic definitions and theorems we have to introduce
	the following denotation.
	Given a semigroup $(S,\cdot)$ one defines the monoid $M(S)$ of $S$ as follows
	\cite{Lallement}:
	$$
	M(S)=\left\{\begin{array}{l l}
		(S,\cdot,1)&\text{if $S$ has an identity $1$;}\\
		(S\cup \{1\},\star,1)&\text{if $S$ does not have an identity; here %
			$x\star y=x\cdot y$,}\\
		& 1\not\in S, x\star 1=1\star x=x, 1\star 1=1\text{ for all } x, y\in S.
		\end{array}\right.
	$$
	 $S^1$  usually denotes the underlying semigroup of $M(S)$.
	
	A structure  $\left(S, \cdot,\bigvee\right)$ is called a {\bf quantale}, if $\left(S, \bigvee\right)$ is a $\bigvee$-semilattice,  $(S, \cdot)$ is a  semigroup, and multiplication distributes over arbitrary joins in both coordinates, that is,
	\begin{equation}\label{eq-1}
		a   \left(\bigvee M\right) = \bigvee (a   M),
	\end{equation}
	and
	\begin{equation}\label{eq-1-2}
		\left(\bigvee M\right)   a = \bigvee (M   a),
	\end{equation}
	for any $a\in S$, $M \subseteq S$.


For every set $X$ we denote by ${\mathcal P}(X)$ the set of all subsets of $X$.
	
\begin{definition}\label{defsel}\em {

A {\bf marked posemigroup} is a pair $(S, {\mathcal A})$, where {$S$ is a posemigroup and}
${\mathcal A}\subseteq {\mathcal P}(S)$,	such that the following conditions hold:
\begin{enumerate}[{\textnormal(\mbox{A}1)}]
\item For all $a\in S$, ${\{a\}}\in {\mathcal A} $.
\item  If $G\in {\mathcal A}$ and  $a, b \in S^1$ then $\{axb \mid x\in G\}\in {\mathcal A} $.
			\end{enumerate}}
		\end{definition}
	
We speak informally about systems {${\mathcal A}$} as {\bf markings} and about
members of 	{${\mathcal A}$} as
{\bf admissible  subsets} of $S$. If necessary, we will use the denotation ${\mathcal A}$-{\bf admissible  subset}. {
	\textbf{Marked posemigroup morphisms} $f : (S, \A) \to(S', \A')$
	are  posemigroup morphisms $f: S \to S'$ that preserve  markings in the sense that $M \in \A \Longrightarrow f(M) \in \A'$.

All marked posemigroups with marked posemigroup morphisms
evidently constitute a concrete category denoted by $\MPOSGR$. There is
 a natural forgetful functor $U$ to the category $\POSG$. }

\begin{proposition}	\label{leftright} The forgetful functor
	$U\colon\MPOSGR\to \POSG$ is faithful and it
	has left and right adjoints.
\end{proposition}
\begin{proof}  Assume that $f,g\colon (S, {\mathcal A})\to (T, {\mathcal B})$ are
	marked posemigroup morphisms such that $Uf=Ug$. Then $f=Uf=Ug=g$.

	For every posemigroup $S$ and every posemigroup morphism
	$f\colon S\to T$, we put
	$E_{\mathbf T}(S)=(S, \{\{x\}\mid x\in S\}), %
	E_{\mathbf P}(S)=(S, \{X\mid X\subseteq S\})$ and
	$E_{\mathbf T}(f)=E_{\mathbf P}(f)=f$.
	
	Clearly, both $E_{\mathbf T}(S)$ and $E_{\mathbf P}(S)$ are marked posemigroups
	and $E_{\mathbf T}(f)$ and $E_{\mathbf P}(f)$ are marked posemigroup morphisms.
	Hence $E_{\mathbf T}$ and $E_{\mathbf P}$ are functors from the category $\POSG$
	to the category $\MPOSGR$.
	
	First, let us prove that $U\dashv E_{\mathbf P}$. Assume that
	$(S, {\mathcal A})$ is a marked posemigroup and $S'$ is a posemigroup.
	Let $f\colon U(S, {\mathcal A}) \to S'$ be a  posemigroup morphism. Then
	$f\colon (S, {\mathcal A}) \to E_{\mathbf P}(S')$ is a marked posemigroup morphism.
	Conversely, let $g\colon (S, {\mathcal A}) \to E_{\mathbf P}(S')$ be a marked posemigroup morphism. Then $g\colon U(S, {\mathcal A}) \to S'$ is
	a  posemigroup morphism. The naturality condition follows trivially.
	
	Similarly, we have  $E_{\mathbf T}\dashv U$.
\end{proof}

\begin{corollary}
	\label{construct}  The category $\MPOSGR$ is a construct with free objects
	and a representable
	forgetful functor 	$\overline{U}$  and hence
	regular wellpowered and regular co-wellpowered.
\end{corollary}
\begin{proof}  Let $X$ be a set and $F(X)$ a  semigroup whose elements are all possible
	non-empty finite sequences of elements of $X$, and the operation consists of placing one sequence after another. All elements of  $F(X)$ are pairwise incompatible,
	i.e. they form an antichain. We put
	${\mathcal C}=\{\{x\}\mid x\in F(X)\}$. Evidently,
	$(F(X), {\mathcal C})$ is a free marked posemigroup over the set $X$.

	Since the category $\POSG$ is a construct,  i.e., a category
	of structured sets and structure-preserving
	functions between them,  and we
	have a faithful functor
	from $\MPOSGR$ to $\POSG$ we obtain a faithful functor
	$\overline{U}$ from  $\MPOSGR$  to $\SET$ such that
	$\overline{U}((S, {\mathcal A}))=S$ and $\overline{U}(f)=f$ for
	every marked posemigroup morphism
	$f\colon (S, {\mathcal A})\to (T, {\mathcal B})$.
	Hence $\MPOSGR$ is a construct such that the forgetful functor $\overline{U}$
	is representable by the free object over a singleton set.
\end{proof}

\begin{proposition}	\label{strongly} The category $\MPOSGR$ is
 strongly complete.
 	\end{proposition}
 \begin{proof}  We have to show that $\MPOSGR$ is complete and has intersections.
 	Since  products in $\POSG$ always exist and coincide with the usual cartesian product
 	equipped with pointwise order and multiplication it is enough to check that a cartesian product of markings will be a marking again. Let
 	$((S_i, {\mathcal A}_i))_{i\in I}$ be a family of marked
 	semigroups indexed by a set $I$. Let us show that $\bigtimes_{i\in I} {\mathcal A}_i$
 	satisfies conditions (A1) and (A2). Assume first that
 	$a=(a_i)_{i\in I} \in \bigtimes_{i\in I} {S}_i$. Then $\{a_i\}\in  {\mathcal A}_i$
 	for every ${i\in I}$. Hence
 	$\{(a_i)_{i\in I}\}=\bigtimes_{i\in I} \{a_i\}\in \bigtimes_{i\in I} {\mathcal A}_i$ and (A1) is satisfied.
 	
 	Assume now that $G=\bigtimes_{i\in I} G_i\in \bigtimes_{i\in I} {\mathcal A}_i$ and  $a=(a_i)_{i\in I}, b=(b_i)_{i\in I} \in (\bigtimes_{i\in I} {S}_i)^1$.
 	Using the pointwise multiplication we obtain that
 	$\{axb \mid x\in G\}\in \bigtimes_{i\in I} {\mathcal A}_i$  and (A2) is satisfied.
 	Evidently, projections are marked posemigroup morphisms and the remaining part is evident.
 	Hence $\MPOSGR$ has products.
 	
 	Equalizers can be constructed as in sets so we omit the proof.

 	 Let  	$((S_i, {\mathcal A}_i))_{i\in I}$ be a family of marked
 	semigroups indexed by a set $I$ which are subobjects of a marked
 	semigroup $(T, \B)$ via monomorphisms $m_i\colon S_i\to T$.
 	 Evidently, monomorphisms in $\MPOSGR$ are injective
 	marked posemigroup morphisms and we can identify them with inclusions.
 	
 	We show that
 	 $(\bigcap_{i\in I} S_i, \bigcap_{i\in I}{\mathcal A}_i)$ is the
 	 intersection of the family $((S_i, {\mathcal A}_i))_{i\in I}$.
 	
 	Then clearly $(\bigcap_{i\in I} S_i, \bigcap_{i\in I}{\mathcal A}_i)$ is a
 	marked posemigroup and the inclusion
 	$m\colon (\bigcap_{i\in I} S_i, \bigcap_{i\in I}{\mathcal A}_i) \to (T, \B)$
 	is an injective marked posemigroup morphism.

 	Let $j\in I$. Then we have
 	the inclusion $f_j\colon (\bigcap_{i\in I} S_i, \bigcap_{i\in I}{\mathcal A}_i)
 	\to (S_j, {\mathcal A}_j)$ which is clearly an injective marked posemigroup morphism.
 	Moreover, $m=m_j\circ f_j$.
 	
 	Now, let $g\colon (U, \C) \to (T, \B)$ be a marked posemigroup morphism such that
 	it factors through each $m_j$, i.e.,  $g=m_j\circ g_j$. Clearly,
 	$g(U)=g_j(U)\subseteq S_j$ for every $j\in I$. Hence
 	$g(U)\subseteq \bigcap_{i\in I} S_i$ and we have $g=m_j\circ \overline{g}$ where
 	$\overline{g}\colon (U, \C) \to  (\bigcap_{i\in I} S_i, \bigcap_{i\in I}{\mathcal A}_i) $ is a corestriction of $g$. Evidently, 	$\overline{g}$ is
 	a  marked posemigroup morphism.
 	\end{proof} 	
	

	\begin{example}\em {
			For a posemigroup $S$, let ${\mathcal D}$ be a collection such that, a subset $M \subseteq S$ is in ${\mathcal D}$  if
		and only if $a(\bigvee M) b$ exists for all $a, b \in S^1$, and satisfies
			\begin{equation}\label{eq-1-3}
			a   \left(\bigvee M\right)   b
				= \bigvee (a   M   b).
				\end{equation}
			It is evident that ${\mathcal D}$ satisfies conditions (A1)-(A2) and hence $(S,\D)$ is a $\D$-marked posemigroup.

	Note that a posemigroup $S$ is a quantale if and only if each subset in $S$ is
${\mathcal D}$-admissible.}
	\end{example}

\begin{remark} \rm Recall that Han, Xia and Gu introduced in \cite{HXG} the notion
	of a strong $m$-distributive subset  $M$ of a posemigroup $S$ that coincides
	with $M$ being an element of ${\mathcal D}$.
\end{remark}	

\begin{definition}\em {
		A marked posemigroup $(S,\A)$ is said to be a \textbf{marked quantale} if every admissible subset has a join, and for any such admissible subset $M$, $M$ satisfies (\ref{eq-1-3}) for any  $a, b \in S^1$.

\textbf{Marked quantale maps} $f \colon (S, \A) \to(S', \A')$ between marked quantales
$(S, \A)$ and $(S', \A')$ are marked posemigroup morphisms $f$
between marked posemigroups $(S, \A)$ and $(S', \A')$  such that $f(\bigvee M)=\bigvee f(M)$ for any admissible subset $M$ of $S$.

A {\bf sub-marked quantale} $(S, \A)$ of a marked quantale $(S', \A')$ is
a marked quantale $(S,\A)$ such that $S\subseteq S'$  and
the inclusion map $i:S\to S'$ is a marked quantale map
$i:(S,\A)\to (S', \A')$.

All marked quantales with marked quantale maps  constitute a category denoted by $\MQUAN$.
}
\end{definition}
	
\begin{remark}
	{\em Related but different axiom systems of admissible sets appear in
	\cite{pasekaf} or \cite{paseka}. Hence we obtain in this paper similar results on reflectivity as in \cite{paseka}.}
\end{remark}

To specify different types of markings we will need the following pullback based approach.	

\begin{definition}\label{defposel}\em {
		
		A {\bf marked poset} is a pair $(S, {\mathcal B})$, where $S$ is a poset and
		${\mathcal B}\subseteq {\mathcal  P}(S)$,	such that the following condition holds:
		\begin{enumerate}[{\textnormal(\mbox{B}1)}]
			\item For all $a\in S$, ${\{a\}}\in {\mathcal B} $.			
	\end{enumerate}}
\end{definition}

 {
 	\textbf{Marked poset morphisms} $f : (S, \B) \to(S', \B')$ are
 	morphisms $f: S \to S'$ of posets that preserve  markings in the sense that $M \in \B \Longrightarrow f(M) \in \B'$.

	All marked posets with marked poset morphisms constitute a category denoted by $\MPOSet$. Then it has a forgetful functor $V$ to the category $\POSet$.

Moreover, there is a forgetful functor $W\colon\MPOSGR\to  \MPOSet$, forgetting the posemigroup structure.
A marked poset  $(S,\B)$ is said to be a \textbf{marked semilattice} if every admissible subset has a join.

\textbf{Marked semilattice maps} $f \colon (S, \B) \to(S', \B')$ between marked semilattices
$(S, \B)$ and $(S', \B')$ are marked poset morphisms $f$
between marked posets $(S, \B)$ and $(S', \B')$  such that $f(\bigvee M)=\bigvee f(M)$ for any admissible subset $M$ of $S$.

All marked semilattices with marked semilattice maps  constitute a category denoted by $\MSEM$.

Let $A\colon \POSet\to \MPOSet$ be a functor such that $A(X)=(X,{\mathcal A})$
and  $A(f)=f$ for all posets $(X,\leqs), (Y,\leqs)$ and all poset morphisms $f\colon X\to Y$. {We point out that such a functor really exists. For instance, one may choose, for an arbitrary poset $(X,\leqs)$, 
\[
\A=\big\{\{x\} \ \big | \ x\in X\big\},
\]
or
\[
\A=\{M\subseteq X\mid M \text{ is a finite set}\},
\]
or
\[
\A=\{M\subseteq X\mid M \text{ is a set whose cardinal number is at most } n\},
\]
here $n$ is an arbitrary positive integer, etc. }

Now consider the diagram
\[
\xymatrix{
	\textbf{$\MPOSGR(A)$} \ar[rr]^{ } \ar[d]_{ }
	&  & \textbf{$\MPOSGR$} \ar[d]^{W}\\
	\textbf{$\POSet$}\ar[rr]_{A}&  & \textbf{$\MPOSet$}}
\]
Here the category $\MPOSGR(A)$ of specified marked posemigroups $(S,{\mathcal A})$ is the ordinary pullback in the underlying 1-category $\textbf{Cat}$: these exist and are unique, by ordinary abstract nonsense.

Explicitly, $\MPOSGR(A)$ is isomorphic to
$\POSet \mathbin{\stackrel{A}{\times}_{\MPOSet}} \MPOSGR$ that
has objects pairs $((X,\leqs), (S,{\mathcal A}))$ such that $A(X,\leqs) = W(S,{\mathcal A})$
and arrows are pairs $(f, g)$ such that $A(f) = W(g)$.
If we choose $A$ so that the images of $A$ and the forgetful functor $W$ are disjoint, we have
$\POSet \mathbin{\stackrel{A}{\times}_{\MPOSet}} \MPOSGR = \emptyset$.

Similarly, there is a forgetful functor $Q\colon\MQUAN\to  \MPOSet$, forgetting the
marked quantale structure.

Sometimes our chosen subsets depend also on the multiplication.
Let $B\colon \POSG\to \MPOSet$ be a functor such that $A(X)=(X,{\mathcal A})$
and  $A(f)=f$ for
all posemigroups $(X,\cdot,\leqs), (Y,\cdot, \leqs)$ and all posemigroup morphisms $f\colon X\to Y$.

Let us consider the diagram
\[
\xymatrix{
	\textbf{$\MQUAN(B)$} \ar[rr]^{ } \ar[d]_{ }
	&  & \textbf{$\MQUAN$} \ar[d]^{W}\\
	\textbf{$\POSG$}\ar[rr]_{B}&  & \textbf{$\MPOSet$}}
\]
As above,  the category $\MQUAN(B)$ of specified marked quantales $(S,{\mathcal A})$ is the ordinary pullback in the underlying 1-category $\textbf{Cat}$.


}

\begin{note}\label{note1}
	{\em Here are some examples of different  {
			markings}		and their corresponding { marked quantales}.
\begin{enumerate}
	
	\item If trivial (one-element) joins are specified, we have
		all posemigroups with posemigroup morphisms. $\POSG$ denotes the respective category.
	\item In the case that ${\mathcal D}$-joins are specified, we obtain
		all posemigroups with posemigroup morphisms  that are ${\mathcal D}$-admissible.
	 		We speak about ${\mathcal D}$-quantales. $\JQUAN$ is a (non-full) subcategory of $\POSG$.
	\item In the case that all joins are specified, we obtain  the notion of
	a quantale. All quantales with join-preserving posemigroup morphisms constitute a full subcategory $\QUAN$ of $\JQUAN$. Moreover,  $\QUAN$ is  a (generally non-full) subcategory of $\MQUAN$.
	\item In the case that finite non-empty joins are specified, we have
	the notion of an m-semilattice.
	\item In the case that joins of upper bounded finite non-empty
	subsetes  are specified, we have the notion of an mb-semilattice.
	\item If (at most) countable joins are specified, we have the notion of a $\sigma$-quantale.
	\item  If joins of subsets with cardinality less
	than some (regular) cardinal $\kappa$ are
	specified, we have the notion of a  $\kappa$-quantale.
	\item In the case that directed joins are specified, we have the notion of a prequantale.
	\item In the case that  joins of chains are specified, we have the notion of a chain-complete prequantale.
	\item In the case that  joins of countable chains are specified, we have the notion of a $\omega$-chain-complete prequantale.
	\item If joins of upper bounded subsets are specified,  we have the notion of a
	bounded-complete quantale.
	\item If joins of upper bounded and directed subsets are specified,  we
	have the notion of a bounded-complete prequantale.
\end{enumerate}	

}
\end{note}

	A subset $X$ of a poset $S$ is a \textbf{lower subset} of  $S$ if $X\da=X$, where
	$X\da=\{p\in S\mid p\leqs x \text{ for some } x\in X\}$. We denote by $\PP(S)$ the set of all lower subsets of the poset $S$.

	A lower subset $D$ of  {a marked poset} {$(S,\A)$} is called an \textbf{$\mathcal A$-ideal} if, for any admissible subset $M$ of $D$, one has that $\bigvee M \in D$.
	Let us denote by $\IDA{(S)}$  the collection of all $\mathcal A$-ideals of $S$.
	If $(S,\A)$ is a marked semilattice then every $s\da$ is an $\mathcal A$-ideal for every  $s \in S$
	and $\IDA{(S)}$ is a complete lattice .

	Moreover,  	if $(S,\A)$ is a marked quantale then $\IDA{(S)}$ is  a quantale (see Proposition \ref{prop1}).
	We then say that $\IDA{(S)}$ is an \textbf{$\mathcal A$-ideal quantale}.
		In particular, $\p(S)$ is a quantale  equipped with the inclusion as ordering and
	\[
	A \cdot B = \{a  b  \mid  a \in A, \ b \in B\}\da
	\]
	as multiplication (\cite{ZL}).
	
		\begin{remark}	{\em
			For a posemigroup $S$, there exists {three} trivial markings over
			$S$, $\A_0 = \{ \{s\} \mid s \in S \}$,  {$\A_1 = \mathcal{P}(S)$} and $\D$.
			It is easy to check that {$Id_{\mathcal{A}_1}(S)\subseteq Id_{\mathcal{D}} (S)\subseteq Id_{\mathcal{A}_0}(S) =\PP(S)$}.
			Moreover, for every marked quantale $(S, \mathcal A)$,
			$Id_{\mathcal{D}} (S)\subseteq Id_{\mathcal{A}}(S)$, and
			a marked posemigroup $(S, \mathcal A)$ is a marked quantale  if and only if
			$\mathcal A\subseteq \mathcal D$.

			One of the referees asked the following question.
			
			\textbf{Question.} Is there a marking $\A$ over a posemigroup $S$ such that
			$(S, \A)$ is a marked quantale, $\A \neq \A_0$ and $\A\neq\D$?
			
			The answer to this question is positive. Namely, for a posemigroup $S$ with at least 3-element subchain let
				\[
				\A=\left\{\{a,b\}\mid a,b\in S \text{ with } a\leqs b\right\}.
				\]
				Then $\A$ is a marking and $\A_0\subsetneq\A\subsetneq\D.$}
	\end{remark}

	\subsection{Reflectors}

	Recall that for a category $\V$ and a subcategory $\U$, a  \textbf{$\U$-reflection} for a $\V$-object $V$ is a $\V$-morphism
	$r:V \to U$ from $V$ to a $\U$-object $U$ with the following universal property:
	for any $\V$-morphism $f: V \to U'$ from $V$ to a $\U$-object $U'$, there exists a unique $\U$-morphism $f': U \to U'$ such that the triangle
	\[
	\bfig
	\qtriangle/->`->`->/[V`U`U';r`f`f']
	\efig
	\]
	commutes (cf. \cite{Joy of cats}).

	Hence we obtain a functor $R:\V\to \U$ such that it is a left adjoint to the inclusion functor. $R$ is called the \textbf{reflector} or \textbf{localization}.
	
	Recall that every reflector $R$ extends to an endofunctor of  $\V$ making $r$ a natural
	transformation.

	\begin{note}\label{note2}
		{\em In the examples listed above in Note \ref{note1}, the
			{$\mathcal A$-ideal quantale}  consists of the following
			{lower} sets $U$ of the ${\mathcal A}$-quantale $S$:
			\begin{enumerate}
				
				\item all {lower} sets $U$, and we have a classical reflection from $\POSG$ to
				$\QUAN$ given by $S\mapsto \p(S)$ (see \cite{R}),
				
				\item all {lower} sets $U$ closed under ${\mathcal D}$-joins,
				\item the principal {lower} sets $U=a\da$, and we have a trivial
				reflection from $\QUAN$ to
				$\QUAN$ given by $S\mapsto \{a\da \mid a\in S\}$ (see \cite{R}),
				\item the ideals of $S$, and we have a
				reflection from m-semilattices to 	$\QUAN$ given by
				$S\mapsto Id(S)$ (see \cite{paseka}),
				\item all {lower} sets $U$  closed under joins of upper bounded
				finite non-empty subsets,
				\item  the $\sigma$-ideals of $S$, and we have a
				reflection from $\sigma$-m-semilattices to 	$\QUAN$ given by
				$S\mapsto \sigma-Id(S)$ (see \cite{paseka}),
				\item  the $\kappa$-ideals of $S$, %
				\item the Scott-closed {lower} sets $U$, and we have a
				reflection from the category of prequantales  	$\QUAN$ given by
				$S\mapsto \Gamma(S)$ (see \cite{paseka}),
				\item all {lower} sets $U$  closed under joins of chains,
				\item all {lower} sets $U$  closed under joins of countable chains,
				\item all {lower} sets $U$  closed under joins of upper bounded subsets, and
				\item all {lower} sets $U$  closed under joins of upper bounded
				directed subsets.
			\end{enumerate}	
			
		}
	\end{note}

	\begin{remark}\rm Recall that one can find another approach to reflectors in quantales
		in \cite{CW2}.		
	\end{remark}

	\section{Quantale structures}\label{section-2}

	{
		Recall (see \cite{R} Definition 3.1.1) that a \textbf{quantic nucleus} on a quantale $Q$ is a
submultiplicative closure operator on $Q$. The subset
\[
Q_j = \{q\in Q\mid j(q)=q\}
\]
can be made into a quantale called a {\bf quantic quotient} of $Q$ as in \cite{R} or \cite{Resende}.

	Let us focus on the nuclei on $\p(S)$ of a given posemigroup $S$.
	Consider a nucleus $j$ on $\p(S)$.  A lower subset $D$ of $S$ is called  a
	{\bf closed subset}, if $D$ is a fixpoint of $j$. In particular, $j$ is said to be {\bf principal closed} if $s\da$ is a fixpoint of $j$ for all $s\in S$. Recall that Han and Zhao in \cite{HZ} say that
	a nucleus $j$ on the power-set quantale $\mathcal P(S)$ is a {\bf topological closure} if $j(\{s\})=s\da$  for all $s\in S$.

As far as we know, there are mainly the following kinds of typical quantic nuclei defined in the set of lower subsets (or power set) of a posemigroup for constructing its injective hull. They were successively presented by Lambek, Barr, Kennison, and Raphael in \cite{LB}, 2012;
Zhang and Laan in \cite{ZL}, 2014; and Xia, Han, and Zhao in \cite{XHZ}, 2017.

\begin{example}\label{example-n1}\em
 Let $A$ be a pomonoid. For $I\in\p(A)$, define
\[
cl(I) = \bigcap\{a\da\backslash b \da/c\da\mid  aIc \subseteq b\da\},
\]
where $c/b = \bigvee\{x \mid xb \leqs c\}$, and $a\backslash c$ defined similarly for $a,b,c\in A$. Then the set $\mathcal{Q}(A)$ of all closed subsets under the quantic nucleus $cl$ is served to be the injective hull for the pomonoid $A$ in the category of pomonoids with submultiplicative order-preserving mappings (\cite{LB}).
\end{example}

\begin{example}\label{example-n2}\em For any lower subset $I$ of a posemigroup $S$, define
\begin{equation}\label{eq-cl}
\cl(I):=\{ x\in S \mid aIc\subseteq b\da \mbox{ implies }
axc\leqs b \mbox{ for all } a,b,c\in S\}.
\end{equation}
Then the mapping $\cl$ is a quantic nucleus on the quantale
$\PP(S)$, and the relative quotient  $\PP(S)_{\cl}$ is the injective hull for the posemgroup $S$ that satisfying $\cl(s\da)=s\da$, $\forall s\in S$, in the category $\POSGRL$ of posemigroups with submultiplicative order-preserving mappings (\cite{ZL}).
\end{example}

\begin{example}\label{example-n3}\em 	Given a  posemigroup $S$ and a lower subset $D\subseteq S$, consider
\begin{equation}\label{eqPOSG6-1}
\ov{D} =\left \{x\in S \ \big | \  \left(\forall a\in S, \ b,c\in S^1\right) \	b  D c\subseteq a\da\Longrightarrow bxc\leqs a\right\}.
\end{equation}
We call $\ov D$ {\bf the closure} of $D$ and say that $D$ is  closed if $\ov{D}=D$. Then ${}^{-}$ is a principal closed quantic nucleus on $\p(S)$. Moreover, $\p(S)_{-}$, which we denote by $\QQ{(S)}$, is the injective hull of $S$ in the category $\POSGRL$  by  \cite{Erratum} Theorem 3.3.

$\ov{D}$ defined in (\ref{eqPOSG6-1}) is an immediate improvement of $\cl(D)$ defined in (\ref{eq-cl}). In fact, we have noticed and announced this verification when we constructed the injective hull for $S$-semigroups in \cite{ZP1}.
\end{example}

\begin{example}\label{example-n4}\em Nuclei on the power set $\mathcal{P}(S)$ of a given posemigroup $S$ are considered in \cite{XHZ} by Xia, Han and Zhao, where for $X\subseteq S,$
\[
X^{\star}=X^{\text{ul}}\cap X^L\cap X^R \cap X^T,
\]
where
\[
X^{\text{ul}}\stackrel{\vartriangle}{=}\{s\in S: \ \forall b\in S, \ X\subseteq b\da\Rightarrow s\leqs b\},
\]
\[
X^L\stackrel{\vartriangle}{=}\{s\in S: \ \forall a,b\in S, \ Xa\subseteq b\da\Rightarrow sa\leqs b\},
\]
\[
X^R\stackrel{\vartriangle}{=}\{s\in S: \ \forall a,b\in S, \ aX\subseteq b\da\Rightarrow as\leqs b\},
\]
\[
X^T\stackrel{\vartriangle}{=}\{s\in S: \ \forall a,b,c\in S, \ aXc\subseteq b\da\Rightarrow asc\leqs b\}.
\]
\end{example}	

We have to point out that the fundamental quantales $\p(S)$ considered in Examples \ref{example-n1}-\ref{example-n3} are different from $\mathcal{P}(S)$ in Example \ref{example-n4}, as they have different underlying posemigroups. Nevertheless,
the set $X^{\star}=\overline{X^{\star}}=%
\overline{\left(\bigcup\limits_{x \in X} (x \da)\right)}=%
{\left(\bigcup\limits_{x \in X} (x \da)\right)}^{\star}\in \p(S)$ and therefore
 both quantic quotients from Examples \ref{example-n3} and \ref{example-n4} are
isomorphic.
}

	{We will look more closely at the properties of
		principal closed nuclei on $\PP(S)$.}
	
\begin{proposition}\label{prop-topol}
Let $S$ be a posemigroup, $j$ be a principal closed nucleus on $\PP(S)$	and $D\in \PP(S)$.  If $\bigvee D$ exists and $\bigvee D \in j(D)$, then $j(D) =\left (\bigvee D\right) \da$ and $\bigvee j(D) = \bigvee D.$
\end{proposition}
\begin{proof}
By the fact that $D\subseteq(\bigvee D)\da$ and $j$ is  principal closed we obtain that
\[
j(D)\subseteq j\left(\left(\bigvee D\right)\big\downarrow\right)=\left(\bigvee D\right)\big\downarrow\subseteq j(D)\da=j(D).
\]
\end{proof}

On the other side, a nucleus $j$ on $\PP(S)$ satisfying $j(D) =\left (\bigvee D\right) \da$ will provide a representation for a quantale.
	
\begin{proposition}\label{prop-represent.}
Let $S$ be a quantale, $j$ a nucleus on $\PP(S)$ satisfying $j(D)=\left(\bigvee D\right)\da$ for any $D\in\PP(S)$. Then $S \cong \PP(S)_j$ as quantales.
\end{proposition}
	
\begin{proof} The statement follows by \cite{R} Theorem 3.1.2.
\end{proof}

Proposition \ref{prop-topol} and Proposition \ref{prop-represent.} contribute to the following essential  results concerning  principal closed nuclei.

\begin{proposition}\label{new-prop}
Let $S$ be a posemigroup and $j$ be a  principal closed nucleus on $\PP(S)$. Then the mapping $\eta:S\to\PP(S)_j$ given by $\eta(s)=s\da$, $\forall s\in S$, has the following properties:
\begin{enumerate}
\item $\eta$ is an order embedding.
\item $\eta$ preserves all existing meets in $S$.
\item $\eta(S)$ is join-dense in the lattice $\PP(S)_j$.
\end{enumerate}
\end{proposition}
	
\begin{proof}
(1) $\eta$ being an order embedding is obvious.
		
(2) Assume that  $\bigwedge M$ exists in $S$ for $M\subseteq S$. It is easy to see that $(\bigwedge M )\da = \bigcap\limits_{m \in M} (m\da)$. The result hence follows by the fact that infs in $\PP(S)$ are intersections.

(3)  Let $D \in \PP(S)_j$. Then
\[
D = j(D) = j(D \da) = j\left(\bigcup\limits_{d \in D} (d \da)\right) = \bigvee \eta (D).
\]
This proves that $\eta(S)$ is join-dense in $\PP(S)_j$.
\end{proof}

In the rest of the paper, we mainly concentrate on two kinds of quantic quotients of $\PP(S)$ whose induced nuclei satisfy the conditions of Proposition \ref{new-prop}.
	

{ One of the quantic quotient of $\p(S)$  is $\QQ{(S)}$ we presented in Example \ref{example-n3},} another quantic quotient we will focus on is  $\IDA{(S)}$. We will need the following well-known statement.

\begin{lemma}\label{lemma-1}  {\em(\cite{R} Proposition 3.1.2)}\label{lemm2}
Let $Q$ be a quantale and $S \subseteq Q$. If $S$ is closed under meets and $a\rightarrow_r s$ and $a\rightarrow_l s$ are in $S$, whenever $a \in Q$ and $s \in S$, then $S = Q_j$ for the quantic nucleus $j$ on $Q$, defined by $j(a)=\bigwedge\{s\in S\mid a\leqs s\}$.
\end{lemma}
	
\begin{proposition}\label{prop1} Let $(S, {\mathcal A})$ be a marked quantale. Then $\IDA{(S)}$ is a  quantic quotient of $\p(S)$.
\end{proposition}
	
\begin{proof} $\IDA{(S)}$ is closed under intersections. By Lemma \ref{lemma-1}, we are going to show that $C\rightarrow_r D$ and $C\rightarrow_l D$ are in $\IDA{(S)}$ for any
$C \in \p(S)$, $D \in \IDA{(S)}$. We only prove that $C\rightarrow_r D\in\IDA{(S)}$. The other case can be similarly obtained.
		
		Notice that $C\rightarrow_r D=\bigcup \{ A \in \p(S) \mid C \cdot A \subseteq D \}$. Write $K=\bigcup \{ A \in \p(S) \mid C \cdot
		A \subseteq D \}$. It is easy to see that $K$ is a lower subset.
		Let $E\subseteq K$ with $E$ being admissible.
		We will prove that $\bigvee E\in K$, i.e.,
		$C\cdot\bigvee E\subseteq D$. Since  $D$ is a lower set, we only need to show that
		$c(\bigvee E)\in D$ for any $c\in C$. Since $E$ is
		admissible, we immediately get by (A2) that $c E$ is admissible.
		By the fact that $c E\subseteq D$ and $D$ is an $\A$-ideal
		we have $c(\bigvee E)=\bigvee (c E) \in D$. We conclude that $\bigvee E \in K$.
		Hence $K\in\IDA{(S)}$.
	\end{proof}

By Lemma \ref{lemma-1} we can find a quantic nucleus $\jj: \p(S)\to \p(S)$, defined by
\begin{equation}\label{eq-2}
\jj(C) = \bigcap \{ D \in \IDA{(S)} \mid C \subseteq D \},
\end{equation}
for every $C \in \p(S)$ such that $\IDA{(S)} = \p(S)_{\ \jj{} \ }$.	We obtain that $\IDA{(S)}$ is a quantale whose multiplication is
\[
V \cdot_{\IDA{(S)}} W =\jj(V\cdot W)= \bigcap \{ D \in \IDA{(S)} \mid V \cdot W \subseteq D \},
\]
for any $V, W \in \IDA{(S)}$, with join being
\[
\bigvee_{i \in I} D_i = \jj\left(\bigcup_{i \in I}D_i\right) =\bigcap \left\{ D \in \IDA{(S)} \ \Bigg | \ %
	\bigcup_{i \in I} D_i \subseteq D \right\},
\]
and meet being the intersection, for any family of $\mathcal A$-ideals $D_i$ of $S$,  $i \in I$.
	
\begin{proposition}\label{prop-topo-D}
For a 	marked quantale $(S, {\mathcal A})$,
the quantic nucleus $ \jj: \p(S)\to \p(S)$ defined in (\ref{eq-2}) is principal closed.
\end{proposition}

\begin{proof}
We have to show that $\jj(s\da)=s\da$ for each $s\in S$. This is equivalent to proving $s\da \in \IDA{(S)}$. But for each  admissible subset $M$ in $s\da$, it is evident that $m \leqs s$ for every $m \in M$. Therefore, $\bigvee M \leqs s$ and then $\bigvee M \in s\da$.
\end{proof}

A natural result for the nucleus $\mathsf{j}_{\D}$  on $\p(S)$ is then obtained by Proposition \ref{prop-represent.}.
	
\begin{corollary}\label{lemm-n2-D}
Let $S$ be a quantale. Then $S$ is isomorphic to the quantale $Id_{\mathcal D}{(S)}$.
\end{corollary}

\section{A reflector for  $\QUAN$ induced by $\AQUAN$}\label{section-3}
	
We are ready to present a reflector for the category of quantales according to  $\A$-ideals.
	
\begin{theorem}\label{theo-D-5}
Let   {$(S,\A)$} be a {marked} quantale. Then the mapping {$t: (S,\A) \to (\IDA{(S),\mathcal{P}({\IDA(S)}))}$}
defined by $t(s) = s \da$, $\forall s \in S$ is a {$\MQUAN$}-morphism, which is also an order-embedding and preserves all existing meets in $S$. In addition, $t(S)$ is join-dense in the lattice $Id_{\mathcal{A}}(S)$.
\end{theorem}

{
\begin{proof}
We only need to show that $t$ is a {$\MQUAN$}-morphism.
The other assertions follow 	directly by Proposition \ref{new-prop}.
		
To verify that $t$ is a {$\MQUAN$}-morphism,
we start from showing that $t$ preserves the multiplication, i.e., $t(a b)=t(a) \cdot_{\IDA(S)}  t(b)$ for any $a, b \in S$. Observe that
\[
t(a) \cdot_{\IDA{(S)}} t(b) = \bigcap \{ I \in \IDA{(S)} \mid (a \da) \cdot (b \da) \subseteq I \},
\]
and $(a \da) \cdot (b \da)= (a b) \da$. So $(a b) \da$ is   one of the terms in the intersection that defines $t(a) \cdot_{\IDA{(S)}} t(b)$. Hence $t(a) \cdot_{\IDD{(S)}} t(b)\subseteq  (ab)\da=t(ab).$
Conversely, we have
\[
t(a b)=t(a)t(b)\subseteq  \mathsf{j}_\A(t(a)t(b))=t(a)\cdot_{\IDA{(S)}} t(b),
\]
which   turns out that  $t(a b)=t(a) \cdot_{\IDA(S)}  t(b)$.

Next, we calculate  $t\left(\bigvee M \right) = \bigvee\limits^{\IDA(S)} t(M)$ for any admissible subset $M$ of $S$. Since
\[
\bigvee  t(M)=\bigcap \left\{ I \in \IDA{(S)} \ \Bigg | \ \bigcup\limits_{m\in M} (m \da)\subseteq I \right\},
\]
it is easy to see that $ \bigvee  t(M) \subseteq t\left(\bigvee M \right)$. Let
$I \in \IDA{(S)}$ and
 $\bigcup\limits_{m \in M} (m \da) \subseteq I$. Then $M \subseteq I$
 and thus $\bigvee M\in I$. As a consequence, we obtain
\[
t\left(\bigvee M\right)=\left(\bigvee M\right)\Big \downarrow\subseteq %
\bigcap \left\{ I \in \IDA{(S)} \ \Bigg | \ \bigcup\limits_{m\in M} (m \da)\subseteq I \right\}\subseteq \bigvee  t(M).
\]	
		
Evidently, $t$ preserves the marking. As a result, we finally arrive that $t$ is a {$\MQUAN$}-morphism.
\end{proof}}
	
\begin{remark}\label{remark-1}\em The order embedding $t: S \to \IDA{(S)}$ in
	Theorem \ref{theo-D-5} fails to preserve all joins that exist in $S$, and the set $t(S)$ is
	not inf-dense in the lattice $\IDA(S)$, see Example \ref{example-1}.
	\end{remark}

The following theorem shows that $\IDA{}$ is left adjoint to the forgetful functor from quantales to marked quantales. In other words, $\IDA{(S)}$  is the free quantale over the	marked quantale $(S,\A)$.
	
	\begin{theorem}\label{theo-reflection-D}
		Let $(S,\A)$ be  a marked quantale.  Then the morphism $t:S \to \IDA{(S)}$,  defined as in Theorem \ref{theo-D-5}, is a $\QUAN$-reflection for $(S,\A)$.
	\end{theorem}

	\begin{proof}Let $Q$ be a quantale, $f:S\to Q$ a marked quantale morphism. We are going to show that there exists a unique $\QUAN$-morphism $g: \IDA{(S)} \to Q$ such that the following diagram commutes.
		\[
		\bfig\qtriangle/->`->`-->/[S`\IDA{(S)}`Q;t`f`g]\efig
		\]
		
		Define $g: \IDA{(S)} \to Q$ by $g(D)= \bigvee  f(D)$, for every $\A$-ideal $D$ of $S$.
		We claim that $g$ is a $\mathsf{Quant}$-morphism. For this purpose, we first show that for any $\A$-ideals $D_1, D_2$ of $S$, we have $g(D_1) g(D_2)=g\left(D_1 \cdot_{\IDA{(S)}} D_2\right)$. Observe that
		\[
		g\left(D_1 \cdot_{\IDA{(S)}} D_2\right)=\bigvee f\left(D_1 \cdot_{\IDA{(S)}} D_2\right),
		\]
		where $D_1 \cdot_{\IDA{(S)}} D_2=\bigcap \{ D \in \IDA{(S)} \mid D_1 \cdot D_2 \subseteq D \}$, and
		\[
		g(D_1)   g(D_2) = \bigvee  f(D_1 D_2),
		\]
		as $Q$ being a quantale and $f$ preserves multiplication. Obviously, $D_1   D_2 \subseteq D_1 \cdot_{\IDA{(S)}} D_2$. Hence, we conclude
		$
		g(D_1) g(D_2) \subseteq g(D_1 \cdot_{\IDA{(S)}} D_2).
		$
		
		To prove the remaining inclusion, consider the subset
		\[
		\kk=\left\{ x \in S \ \Big | \ f(x) \leqs \bigvee  %
		f(D_1 \cdot D_2)\right\}
		\]
		of $S$. Evidently, $\kk$ is a lower subset of $S$. Let
		$M \subseteq \kk$ with $M$ being admissible. Hence
		$f(\bigvee M)=\bigvee f(M)\leqs \bigvee f(D_1 \cdot D_2)$, i.e.,
		$\bigvee M \in \kk$.  We conclude that 	$\kk$ is an $\A$-ideal of $S$.

		Apparently, $D_1 \cdot D_2 \subseteq \kk$. This indicates that $D_1 \cdot_{\IDA{(S)}} D_2\subseteq \kk$ since $\kk$ is an $\A$-ideal. It hence follows that
		\[
		g(D_1 \cdot_{\IDA{(S)}} D_2) =\bigvee  f(D_1 \cdot_{\IDA{(S)}} D_2)
		\leqs \bigvee  f(\kk)
		\leqs \bigvee f(D_1 \cdot D_2)
		= g(D_1) g(D_2).
		\]
		We conclude $g(D_1 \cdot_{\IDA{(S)}} D_2) = g(D_1) g(D_2)$.
		
		Assume that $ D_i, \ {i \in I}$ is a family of $\A$-ideals of $S$. Then $\bigvee\limits_{i \in I}^{\IDA(S)} D_i$ exists.
		Since
		\[
		\bigvee_{i\in I} g(D_i) = \bigvee_{i\in I} \bigvee f(D_i)=\bigvee  f\left(\bigcup\limits_{i \in I} D_i\right),
		\]
		and
		\[
		g\left(\bigvee_{i \in I}^{\IDA(S)}  D_i\right) = \bigvee  f\left(\bigvee_{i \in I}^{\IDA(S)}  D_i\right),
		\]
		the inequality $\bigvee\limits_{i \in I} g(D_i) \leqs g\left(\bigvee\limits_{i \in I}^{\IDA(S)}  D_i\right)$ is clear. Consider the set
		\[
		\widetilde{\kk}=\left\{ x \in S \ \Bigg | \ f(x) \leqs \bigvee f\left(\bigcup\limits_{i \in I} D_i\right)\right\}.
		\]
		A similar proof as we did for $\kk$ above shows that
		$\widetilde{\kk}$ is an $\A$-ideal as well and contains
		the set $\bigcup\limits_{i \in I} D_i$ and therefore
		also $\bigvee\limits_{i \in I}^{\IDA(S)} D_i$. Hence we get
		that $g\left(\bigvee\limits_{i \in I}^{\IDA(S)} D_i\right) %
		\leqs g\left(\widetilde{\kk}\right)\leqs \bigvee f\left(\bigcup\limits_{i \in I} D_i\right)= %
		\bigvee\limits_{i \in I} g(D_i)$. As a result, we obtain that $g$ is a $\QUAN$-morphism. For any $s \in S$,
		\[
		g(t(s)) = g(s \da) = \bigvee f(s \da) = f(s),
		\]
		gives us that $g\circ t=f$.
		
		It remains to show that $g$ is unique. Suppose that
		$h: \IDA{(S)} \to Q$ is a $\QUAN$-morphism such that $h\circ t = f$.
		Then $h(s \da)=f(s)$ for every $s \in S$.
		
		We observe that
		$
		D = \bigvee\limits_{d \in D} (d \downarrow)
		$
		in $\IDA{(S)}$ for any $\A$-ideal $D$ of $(S,\A)$.  Hence
		\begin{align*}
			g(D) & = \bigvee  f(D) = \bigvee\limits_{d \in D} h(d \downarrow) = h \left(\bigvee\limits_{d \in D} (d \downarrow)\right) = h(D).
		\end{align*}
		As a consequence, we eventually achieve that $t$ is a $\QUAN$-reflection
		for $(S,\A)$.
	\end{proof}
	
	Next we show that the process of taking $\A$-ideals is functorial into the category
	$\QUAN$ of quantales.

	\begin{proposition}\label{theo-reflector-D(S)}
		There is a  reflector functor $Id:\MQUAN\to\QUAN$ constructed according to the assignment
		\[
		\xymatrix{
			(S,\A) \ar[rr]^{t_S} \ar[d]_{f}
			&  & \IDA(S) \ar[d]^{Id(f)}\\
			(T,\B) \ar[rr]^{t_T}&  & \IDB(T)}
		\]
		where
		\[
		Id(f)(D) = \jbj(f(D)\da)
		\]
		for every $D \in \IDA(S)$.
	\end{proposition}
	\begin{proof}
		By Theorem \ref{theo-reflection-D}, we know that $\QUAN$ is a reflective subcategory of $\MQUAN$, and  the morphism $t:(S,\A) \to \IDA{(S)}$  is a $\QUAN$-reflection for each marked quantale $(S,\A)$. Thus the inclusion functor $E\colon\QUAN \to \MQUAN$ defined as $E(Q)=(Q,{\mathcal P}(Q))$ and $E(f)=f$ has a left adjoint $Id: \MQUAN \to \QUAN$ which, by proposition 4.22 in \cite{Joy of cats}, can be described explicitly as follows. It takes an
		object $(S,\A)$ of $\MQUAN$ to an object $\IDA{(S)}$ of $\QUAN$ and a morphism $f: (S,\A) \to 	(T,\B)$ in $\MQUAN$ to the unique morphism $Id(f): \IDA(S) \to \IDB(T)$
		in $\QUAN$ such that the square
		\[
		\xymatrix{
			S \ar[rr]^{t_S} \ar[d]_{f}
			&  & \IDA(S) \ar[d]^{Id(f)}\\
			T \ar[rr]^{t_T}&  & \IDB(T)}
		\]
 commutes. By Theorem \ref{theo-D-5}, $t_S(S)$ is join-dense in $\IDA(S)$.
		We conclude
		\begin{align*}
			Id(f)(D) &= Id(f) \left( \bigvee t_S(D) \right) = %
			\bigvee  Id(f) \left( t_S(D) \right) = \bigvee  t_T \left( f(D) \right)\\
			&= \bigvee_{d \in D} f(d)\da = \jbj\left(\bigcup_{d \in D} f(d)\da\right) = \jbj(f(D)\da),
		\end{align*}
		for each $D\in\IDA(S)$.
	\end{proof}

\begin{remark}\rm In what follows we will describe units and counits of the adjunction betweeen $Id$ and $E$. For each object 	$(S,\A)$  in $\MQUAN$ there is a morphism
	$\eta_{(S,\A)}\colon (S,\A) \to (\IDA(S),{\mathcal P}(\IDA(S)))$ (unit) defined as
	$\eta_{(S,\A)}(x)=x\da=t(x)$. Since $t$ is a reflection for $(S,\A)$ then
	for every quantale $Q$ and every $\MQUAN$-morphism $f\colon (S,\A) \to E(Q)$
	there is
	a unique morphism of quantales $g\colon \IDA(S)\to Q$ such that
	 the following diagram commutes
		\[
	\bfig\scalefactor{1.4}%
	\qtriangle/->`->`-->/[(S,\A)`E(\IDA{(S)})`E(Q);\eta_{(S,\A)}`f`E(g)]\efig
	\]
	Similarly, for each quantale $Q$ there is a quantale morphism
	$\varepsilon_Q\colon Id_{{\mathcal P}(Q)}(Q)\to Q$ (counit) defined as
	$\varepsilon_Q(a\da)=a$ since $Id_{{\mathcal P}(Q)}(Q)=\{a\da \mid a\in Q\}$.
	Moreover, for each object 	$(S,\A)$  in $\MQUAN$ and every quantale morphism
	$h\colon \IDA(S) \to Q$ there exists a unique $\MQUAN$-morphism
	$k\colon (S,\A)\to E(Q)$ such that
	the following diagram commutes
		\[
	\bfig\scalefactor{1.4}%
	\qtriangle/-->`->`->/[\IDA(S)`Id_{{\mathcal P}(Q)}(Q)%
	`Q;Id(k)`h`\varepsilon_Q]\efig
	\]
	where $k(s)=h(s\da)$ for all $s\in S$. Recall that $Q$ and $Id_{{\mathcal P}(Q)}(Q)$
	are isomorphic as quantales.
\end{remark}

\begin{remark}\rm  Recall that all our results remain valid if multiplication on
	$(S,\A)$ is commutative or it coincides with the meet operation. The respective quantale $\IDA(S)$ is then commutative or a frame \cite[Proposition 3.4]{frith2}, respectively, the prescription $Id$ is functorial  \cite[Proposition 3.5]{frith2}
	and the mapping $t$ yields a reflection to commutative quantales
	or frames \cite[Proposition 3.8]{frith2}, respectively.
	
\end{remark}

	
\section{A reflector for $\QUAN$ induced by $\CPS$}

In this section  we discuss another reflection to quantales from
the category of posemigroups where morphisms are closure and multiplication preserving mappings.
Some parts of our results were obtained in an equivalent setting by Xia, Zhao and Han in \cite[Lemma 3.11, Proposition 3.13 and Theorem 3.14]{XZH}. For the reader's convenience
(the paper is in Chinese) we include all proofs.
	
Recall that we have mentioned in Section \ref{section-2} that for a posemigroup $S$, the principal closed nucleus ${}^{-}$ on $\PP(S)$ contributes to give a quantic quotient $\QQ(S)$ and  induces a natural order-embedding $\tau:S\to\QQ(S)$ by $\tau(s)=s\da$ for every $s\in S$.

We say that a morphism $f:S\to T$ between posemigroups is {\bf closure preserving} if
\[
f\left(\ov{M\da}\right)\subseteq\ov{f(M)\da}
\]
for every subset $M$ of $S$ (see Example \ref{example-1}). Let us write $\CPS$ for the subcategory of $\POSG$, whose objects are posemigroups and morphisms are closure and multiplication preserving mappings. Thanks to Proposition \ref{new-prop},  we obtain an analogous result for the quantale $\QQ(S)$ similar to Theorem \ref{theo-D-5}.
	
\begin{theorem}\label{theo-tau}
Let $S$ be a posemigroup. Then the mapping $\tau: S \to \QQ{(S)}$ defined by $\tau(s) = s \da$, $\forall s \in S$, is a $\CPS$-morphism, which is also an order-embedding and preserves all existing meets in $S$. In addition, $\tau(S)$ is join-dense in the lattice $\QQ(S)$.
\end{theorem}
	
\begin{proposition}\label{prop-Q-D}
Let $(S,\A)$ be a marked quantale. Then  $\QQ(S)\subseteq\IDA(S)$.
\end{proposition}
\begin{proof}
Take $D \in \QQ(S)$. Since $D$ is a lower set we only have	to show that, for any admissible subset $M$  of $D$, we have $\bigvee M \in D$. Assume that $b(M\da)c\subseteq a\da$ for $b, c \in S^1, \ a \in S$. Then $bMc\subseteq a\da$ and hence $b\left(\bigvee M\right)c =\bigvee(bMc)\leqs a$ because $M$ is admissible. We conclude $\bigvee M\in \ov{M\da} \subseteq \ov{D}=D$.
\end{proof}

\begin{remark}\em The inclusion $\IDA(S)\subseteq\QQ(S)$ is not generally true, see Example \ref{example-2}.
\end{remark}

As the nucleus ${}^{-}$ is principal closed, compared with Corollary \ref{lemm-n2-D}, we conclude the following results by Proposition \ref{prop-Q-D}.
	
\begin{corollary}\label{lemm-n2}
Let $(S,\A)$ be  {a marked quantale}, $D\subseteq S$ with $\bigvee D$ existing.  Then the following statements hold.
\begin{enumerate}
\item $\ov{D\da}\subseteq(\bigvee D)\da$.
\item If $S$ is a quantale  or $D$ is {an $\A$}-admissible subset then	$\ov{D\da} =\left (\bigvee D\right) \da\in \IDA(S)$	and $\bigvee \ov{D\da} = \bigvee D.$
\item If $S$ is a quantale, then $S \cong \QQ{(S)}$  as quantales. Moreover, $\QQ{(S)}={Id}_{{\mathcal P}(S)}(S)$.
\end{enumerate}
\end{corollary}

	\begin{proposition}
		Let $f:S \to T$ be a morphism of posemigroups. Then the following statements are equivalent:
		\begin{enumerate}
			\item $f$ is closure preserving.
			\item $\ov{f\left(\ov{M}\right)\da} = \ov{f(M)\da}$ for every $M \in\PP(S)$.
			\item $\ov{f^{-1}(M)\da}=f^{-1}(M)$ for every $M \in \QQ(T)$.
		\end{enumerate}
	\end{proposition}
	
	\begin{proof}
		(1)$\Rightarrow$(2) Clearly, $f(M)\da \subseteq f\left(\ov{M}\right)\da$.
		By assumption, $f\left(\ov{M}\right)\subseteq\ov{f(M)\da}$. We conclude
		\[
		\ov{f(M)\da} \subseteq \ov{f\left(\ov{M}\right)\da} %
		\subseteq \ov{\ov{f(M)\da}} = \ov{f(M)\da},
		\]
		that is, $\ov{f(\ov{M})\da} = \ov{f(M)\da}$ as desired.
		
		(2)$\Rightarrow$(3) Suppose that $M \in \QQ(T)$ and $x \in \ov{f^{-1}(M)\da}$. Then
		\[
		f(x) \in f\left(\ov{f^{-1}(M)\da}\right) \subseteq \ov{f\left(\ov{f^{-1}(M)\da}\right)\da} = \ov{f\left(f^{-1}(M)\right)\da} \subseteq \ov{M} = M.
		\]
		So $x \in f^{-1}(M)$, and we achieve that
		$\ov{f^{-1}(M)\da} \subseteq f^{-1}(M) \subseteq \ov{f^{-1}(M)\da}$.
		
		(3)$\Rightarrow$(1)  It is obvious that for each subset $M \subseteq S$,
		$\ov{f(M)\da} \in \QQ(T)$. Since $M \subseteq f^{-1}\left(\ov{f(M)\da}\right)$, we conclude that
		\[
		\ov{M\da} \subseteq \ov{f^{-1}\left(\ov{f(M)\da}\right)\da}%
		=f^{-1}\left(\ov{f(M)\da}\right).
		\]
		Therefore, $f\left(\ov{M\da}\right) \subseteq \ov{f(M)\da}$. The proof is complete.
	\end{proof}
	
	 The following propositions yield that $\QUAN$ is a full subcategory of $\CPS$ and $\CPS$ is a subcategory of $\DQUAN$.

\begin{proposition} \cite[Proposition 3.13]{XZH}\label{prop-new-1}
Let $S,T$ be quantales and $f:S\to T$ be a posemigroup morphism. Then $f$ is a quantale morphism if and only if  it is closure preserving.
\end{proposition}
	
	\begin{proof} Let $D \subseteq S$.
		Suppose that $f$ is closure preserving. Apparently,  $\bigvee f(D)\leqs f(\bigvee D)$.
		Observe that $\ov{f(D)\da} = \left(\bigvee f(D)\right)\da$.
		From Corollary \ref{lemm-n2} we conclude that, for any $t\in S$,
		\[
		\bigvee f(D) \leqs t\Longrightarrow f\left(\ov{D\da}\right)%
		\subseteq\ov{f(D)\da} = \left(\bigvee f(D)\right)\Big\downarrow \subseteq t\da\Longrightarrow f\left(\left(\bigvee D\right)\Big\downarrow\right) \subseteq t\da.
		\]
		Therefore, $f(\bigvee D)\leqs\bigvee f(D)$ as needed.
		
		Conversely, if $f$ is a quantale morphism then
		\[
		f\left(\ov{D\da}\right) \subseteq f\left(\left(\bigvee D\right) \Big \downarrow\right) %
		 \subseteq f\left(\bigvee D \right) \Big \downarrow = \left(\bigvee f(D) \right) \Big \downarrow = \ov{f(D)\da}.
		\]
	\end{proof}

	\begin{proposition}\label{prop-new-2}
		Let $S,T$ be marked quantales. If $f:S\to T$ is closure preserving
		then it preserves all admissible joins.
	\end{proposition}
	\begin{proof}
		Let $D$ be an admissible subset in $S$. Then $\bigvee D$ exists, and
		$(\bigvee D)\da = \ov{D\da}$ by Corollary $\ref{lemm-n2}$.
		Since  $f$ is  closure preserving we conclude
		$f \left(\bigvee D \right) \in f(\ov{D\da}) \subseteq \ov{f(D)\da}$.
		Let $t\in T$ such that $f(D)\subseteq t\da$. Then
		$f(D)\da\subseteq t\da$ and
		$f \left(\bigvee D \right) \in \ov{f(D)\da}$ which implies
		$f(\bigvee D) \leqs  t$.
		We conclude that
		$f(\bigvee D)$ is the least upper bound of $f(D)$, namely, $f(\bigvee D) = \bigvee f(D)$.
	\end{proof}

	\begin{corollary}\label{cor-new-2}
	Let $S,T$ be marked quantales. If $f:S\to T$ is closure and admissible preserving then $f$ is a marked quantale map.
\end{corollary}
	
		\begin{corollary}\label{xcor-new-2}
		Let $S,T$ be ${\mathcal D}$-quantales and the map  $f:S\to T$ be closure  preserving. Then $f$ is a $\D$-quantale map.
	\end{corollary}
	\begin{proof} Let $D$ be a ${\mathcal D}$-admissible subset in $S$.
		It is enough to check that  $f(D)$ is ${\mathcal D}$-admissible in $T$.
		Since $\bigvee D\in  \ov{D\da}$ we obtain
		$f(\bigvee D) \in f(\ov{D\da})\subseteq \ov{f(D)\da}$.
		It is evident that $af(d)b \leqs af(\bigvee D)b$ for any $a,b \in T^1$, $d\in D$. Moreover, if $af(D)b \subseteq t\da$ for some $t\in T$ then $af(\bigvee D)b \leqs t$.
		We hence conclude that $\bigvee \left(af(D)b\right)$ exists and $\bigvee \left(af(D)b\right)=af(\bigvee D)b$. Therefore,
		$$
		a\left(\bigvee f(D)\right)b = af\left(\bigvee D\right)b = \bigvee \left(af(D)b\right).
		$$
			\end{proof}

	\begin{remark}\em Recall that the converse of Corollary \ref{xcor-new-2} is not true, i.e., a morphism of ${\mathcal D}$-quantales need not be closure preserving (see Example \ref{example-3}).
	\end{remark}

	
	\begin{lemma}\label{lemma-new}
		Let $S,T$ be marked quantales. If $f:S\to T$ is closure preserving, then
		$\bigvee f\left(\ov{D\da}\right) = \bigvee \ov{f(D)\da}$
		for any admissible subset $f(D)\subseteq T$.
	\end{lemma}
	\begin{proof}
		
		Since $f\left(\ov{D\da}\right) \subseteq \ov{f(D)\da}$ by the fact that $f$
		is closure preserving,  $\bigvee f\left(\ov{D\da}\right) \leqs \bigvee \ov{f(D)\da}$ follows immediately. It is obvious that $f(D)\subseteq f\left(\ov{D\da}\right)$. Hence
		\[
		\bigvee f\left(\ov{D\da}\right)\leqs \bigvee \ov{f(D)\da}=\bigvee f(D)%
		\leqs \bigvee f\left(\ov{D\da}\right)
		\]
		by Corollary \ref{lemm-n2}, (1). 
	\end{proof}

	As a result, we obtain the following relationship of $\POSG$ and its subcategories.
	\[
	\QUAN\subseteq\MQUAN\subseteq\POSG\quad \text{and \quad}
	\QUAN\subseteq\CPS\subseteq\DQUAN\subseteq\POSG.
	\]

	We are now prepared to describe our next reflector to the category $\QUAN$. Unlike previous reflectors realized by $\A$-ideals, the reflection obtained by closed subsets has another nice property.
	Recall that  the first half of Theorem \ref{theo-1} is contained in
	\cite[Theorem 3.14]{XZH}.

	\begin{theorem}   \label{theo-1}
		Let $S$ be a posemigroup.  Then the morphism $\tau:S \to \QQ{(S)}$ defined as in Theorem \ref{theo-tau} is a full $\QUAN$-reflection for $S$. Moreover,
		if $Q$ is a quantale and $f:S\to Q$ is an order-embedding in $\CPS$, then so is
		$g: \QQ(S) \to Q$ such that $f=g\circ \tau$.
	\end{theorem}
	
	\begin{proof}
		Let $Q$ be a quantale, $f:S\to Q$  a $\CPS$-morphism. The procedure is to find a unique $\QUAN$-morphism $g: \QQ(S) \to Q$  such that the following diagram commutes.
		\[
		\bfig
		\qtriangle/->`->`-->/[S`\QQ(S)`Q;\tau`f`g]
		\efig
		\]

		Define a mapping $g: \QQ(S) \to Q$ by
		\[
		g(D):= \bigvee f(D), \
		\]
		for $ D\in \QQ(S).$ We then obtain that
		\[
		g(\tau(s)) = g(s\da) = \bigvee f(s\da) = \bigvee (f(s)\da) = f(s),
		\]
		for every $s \in S$.

		We show that $g$ preserves arbitrary joins. Given  a family of subsets $D_i$, $i \in I$, in $\QQ(S)$. Then $\bigvee\limits_{i \in I} D_{i} = \ov{\bigcup\limits_{i \in I} D_{i}}$. 
		By Lemma \ref{lemma-new}, we calculate that
		\begin{align*}
			g\left(\bigvee_{i \in I} D_{i}\right) &= g\left(\ov{\bigcup_{i \in I} D_{i}}\right)= \bigvee f\left(\ov{\bigcup_{i \in I} D_{i}}\right)
			= \bigvee \ov{f\left(\bigcup_{i \in I} D_{i}\right)\da}\\
			&= \bigvee \ov{\left(\bigcup_{i \in I} f(D_{i})\right)\da}
			= \bigvee \bigcup_{i \in I} f(D_{i})
			= \bigvee_{i \in I} \left( \bigvee f(D_{i}) \right)= \bigvee_{i \in I} g(D_{i}).
		\end{align*}
		In the same manner, one can see that $g$ preserves multiplications.
		The uniqueness of $g$ is straightforward. From Proposition \ref{prop-new-1} we obtain that
		the reflection is full.

		It remains to show that $g$ is an order-embedding whenever $f$ is an order-embedding. Evidently, $g$ is monotone. Suppose that $D,E\in \QQ(S)$ with $g(D)\leqs g(E)$. Our goal is to prove that $g$ reflects the order. It is enough to check that $D \subseteq \ov{E}$ since $\ov{E}=E$. Consider $a E b\subseteq c\da$, where $a, b \in S^1$, $c\in S$. We have $ f(a) f(E) f(b) \subseteq f(c)\da$, and so $\bigvee (f(a) f(E) f(b)) \leqs f(c)$. Therefore,
		\[\begin{array}{r c l}
		f(adb)&=&f(a) f(d) f(b)\leqs f(a) \left( \bigvee f(D)\right) f(b)= f(a) g(D) f(b)\\[0.2cm]
		&\leqs&  f(a) g(E) f(b)= f(a) \left( \bigvee f(E)\right) f(b)=\bigvee (f(a) f(E) f(b))\leqs f(c)
		\end{array}
		\]
		for every $d \in D$. Since $f$ is an order-embedding, we obtain $adb \leqs c$,  which means that $D \subseteq \ov{E}$ as desired.
	\end{proof}

	\begin{remark}\em Recall that the order-embedding property that holds in Theorem \ref{theo-1} does
		not hold for the $\QUAN$-reflection obtained by $\D$-ideals in Theorem  \ref{theo-reflection-D}, see Example \ref{example-2}.
	\end{remark}

	By a similar method as in Proposition \ref{theo-reflector-D(S)}, we thereby obtain another reflector for category $\QUAN$.
	
	\begin{proposition}\cite[Theorem 3.14]{XZH}\label{coro-1}
		$\QUAN$ is a full reflective subcategory of $\CPS$ with the reflector functor $\QQ: \CPS \to \QUAN$ defined by the assignment
		\[
		\xymatrix{
			S \ar[rr]^{\tau_S} \ar[d]_{f}
			&  & \QQ(S) \ar[d]^{\QQ(f)}\\
			T \ar[rr]^{\tau_T}&  & \QQ(T)}
		\]
		where
		\[
		\QQ(f)(D) = \ov{f(D)\da}
		\]
		for every $D \in \QQ(S)$.
	\end{proposition}


	\section{Examples}

	In this section, we present, for illustration purposes, five simple examples.
	
	\begin{example}\label{exam-qu-n-a}\rm
		Let $X=\{x,y,z\}$ be a poset with $z$ being the top element,  $x$ and $y$ being parallel. Let
		\[
		S=\{ a_1 x_1 a_2 x_2 a_3 \cdots a_n x_n a_{n+1}\mid x_i \in  X, \  a_j \in \mathbb{N}, \ i\in\{0,1,\ldots ,n\},\ j\in\{1,2,\ldots ,n+1\}, \ n \in \mathbb{N} \}.
		\]
		For $\alpha\in S$, denote by $|\alpha|$ as the number of the elements $x_i $ that appear in $\alpha$. For example, $|0|=0$, $|3x1|=1$, $|2z3y6|=2.$ For convenience, if $\alpha\in S$ with $|\alpha|=n$, we will write
		\[
		\alpha=a_1^{\alpha} x_1^{\alpha} a_2^{\alpha} \cdots a_n^{\alpha} x_n^{\alpha} a_{n+1}^{\alpha}.
		\]
		
		Given $\alpha,\beta\in S$ with $|\alpha|=n$, $|\beta|=m$,  the multiplication $\alpha \cdot \beta$ is defined by
		\[
		\alpha \cdot \beta := a_1^{\alpha} x_1^{\alpha} a_2^{\alpha} \cdots a_n^{\alpha} x_n^{\alpha} \left(a_{n+1}^{\alpha}+a_1^{\beta}\right) x_1^{\beta} a_2^{\beta} \cdots a_m^{\beta} x_m^{\beta} a_{m+1}^{\beta}.
		\]
		It is clear that $(S, \cdot)$ is a semigroup. Define a relation $\leqs$ on $S$ by $\alpha \leqs\beta$ if and only if the following conditions hold:
		
		($S$1) $|\alpha|=|\beta|$,
		
		($S$2) $\alpha =\beta$ when $|\alpha|=0$,
		
		($S$3)  if  $|\alpha|>0$ then $x_i^{\alpha} \leqs  x_i^{\beta}$ for each $i=1,\ldots,|\alpha|$, and
		
		($S$4) there exists a unique sequence $c_1,\ldots,c_{|\alpha|}$ such that the equalities in (\ref{eq-example1}) are satisfied,
		\begin{equation}\label{eq-example1}
			a_j^{\beta} = a_j^{\alpha} + c_{j-1}+c_j,  \   j\in\{1,\ldots, |\alpha|+1\},
		\end{equation}
		with $c_0=0=c_{|\alpha|+1}$, and for $i\not=0,|\alpha|+1$, $c_i\in\N$ if $x_i^{\alpha} = x_i^{\beta}$, and $c_i\in\Z$ otherwise.
		
		It is routine to check that $\leqs$ defined above is a partial order on $S$ and $(S,\cdot,\leqs)$ becomes a posemigroup. For each pair $m\neq n\in \N$,
		$m$ is incomparable to $n$ by (S2). Therefore, 0 is parallel to 1.

		Now, we are going to find a subset $M$ in $S$ that satisfies (\ref{eq-1}) and (\ref{eq-1-2}), but not  (\ref{eq-1-3}).
		
		Take $M=\{0x0, \ 0y0\}$. Clearly, $0z0$ is an upper bound of $M$. Consider an upper bound $\alpha$ of $M$ in $S$.  Then $|\alpha|=1$ by ($S$1) and $x_1^{\alpha}=\bigvee\{x,y\}$ by ($S$3). Therefore $x_1^{\alpha}=z$.
		Moreover, ($S4$) provides that $a_1^{\alpha} =a_2^{\alpha} $,   thus $0z0 \leqs \alpha$. We hence obtain that $\bigvee M = 0z0$.
		
		Next, we prove that $\alpha \cdot (\bigvee M) = \bigvee(\alpha \cdot M)$  for any $\alpha \in S$. Evidently, $\alpha \cdot 0z0$ is an upper bound of $\alpha \cdot M$.
		Assume that $\beta$ is an upper bound of $\alpha \cdot M$. If $|\alpha|=0$, then $\alpha \in\N$ and $|\beta|=1$. It is clear that $x_1^{\beta}=z$.
		Since $\alpha \cdot 0x0 \leqs \beta$, and $a_2^{\alpha \cdot 0x0}=0$, we have  $a_1^{\beta} = a_1^{\alpha \cdot 0x0} + a_2^{\beta}$. Notice that $a_1^{\alpha \cdot 0x0}=a_1^{\alpha \cdot 0z0}$, we immediately obtain that
		\[
		a_1^{\beta}  =a_1^{\alpha \cdot 0z0} +a_2^{\beta}, \ \ a_2^{\beta}=a_2^{\alpha \cdot 0z0}+
		a_2^{\beta},
		\]
		which gives that $\alpha \cdot 0z0 \leqs \beta$ for $|\alpha|=0$.
		
		If $|\alpha|>0$, then $|\beta|=|\alpha \cdot 0x0|=|\alpha|+1$, $ x_{|\beta|}^{\alpha \cdot 0x0}=x$, and $x_{|\beta|}^{\alpha \cdot 0y0}=y$. By ($S$3), we get that $x_{|\beta|}^{\beta}=z$.
		Since $\alpha \cdot 0x0 \leqs \beta$, there exists a unique sequence   $c_1,\ldots,c_{|\beta|}$ satisfying (\ref{eq-example1}), that is,
		\begin{equation}\label{neweq}
			a_j^{\beta} = a_{j}^{\alpha \cdot 0x0} + c_{j-1} + c_{j}, \ j\in\{1,\ldots, |\beta|+1\},
		\end{equation}
		with $c_0=0=c_{|\beta|+1}$, for $i\not=0,|\beta|+1$, $c_i\in\N$ when $x_i^{\alpha \cdot 0x0} = x_i^{\beta}$, and $c_i\in\Z$ otherwise. It is easy to calculate that for the same sequence $c_1,\ldots,c_{|\beta|}$ in (\ref{neweq}), we have
		\[
		a_j^{\beta} =a_{j}^{\alpha \cdot 0z0} + c_{j-1} + c_{j}, \ j\in\{1,\ldots, |\beta|+1\},
		\]
		with $c_i$ fulfilling (\ref{eq-example1}) for $ a_j^{\beta}$ and $a_{j}^{\alpha \cdot 0z0}$.
		So $\alpha \cdot 0z0 \leqs \beta$ for $|\alpha|>0$. Finally, we obtain that $\alpha \cdot (\bigvee M) =\alpha \cdot 0z0= \bigvee(\alpha \cdot M)$. The equality  $(\bigvee M) \cdot \alpha = \bigvee(M \cdot \alpha)$ can be similarly verified.
		
		 {Let us compute $\bigvee(1 \cdot M \cdot 1)$.  We have
			
			$$1 ((0x0)1)=1(0x1)=1x1\quad \text{and}\quad
			1((0y0) 1)=1(0y1)=1y1. $$
			
			Put $\overline{M}=\{1x1, 1y1\}$. Consider an upper bound $\alpha$ of $\overline{M}$ in $S$.  Then $|\alpha|=1$ by ($S$1) and $x_1^{\alpha}=\bigvee\{x,y\}=z$ by ($S$3).
			Moreover, $1x1 \leqs \alpha$ and ($S4$) provide that
			\[
			a_1^\alpha=1+0+c_1 \quad \text{and}\quad a_2^\alpha=1+c_1+0,
			\]
			{where $c_1\in \Z$.}

			Hence $a_1^{\alpha} =a_2^{\alpha}=1+c_1\in \N$ 	{(since $x_1^{\alpha}=z\not= x$  we have that  $c_1$ need not be in $\N$)}. Hence every upper bound of $\overline{M}$ is
			of the form $kzk$ where $k\in \N$. It is enough to put $c_1=k-1$. In particular $0z0\geqslant 1x1$ and $0z0\geqslant 1y1$.
						}
		
		{We now calculate that $0z0\leqs kzk$ for all $k\in \N$. If we put $(c_0,c_1,c_2)=(0,k,0)$ we obtain
			\[
			k=0+0+k\quad \text{and} \quad k=0+k+0.
			\]
		
		It is now clear that
		\[
		1\cdot\left(\bigvee M\right)\cdot 1=1z1>0z0=\bigvee(1 \cdot M \cdot 1).
		\]
}	
	\end{example}

{
\begin{example}\label{example-n5}\em Let $X=\{1,2,3\}$ and $S=\mathcal{P}(X)=\{u,a,b,c,d,e,f,v\}$, where $u=\emptyset$, $a=\{1\}$, $b=\{2\}$, $c=\{3\}$, $d=\{1,2\}$, $e=\{1,3\}$, $f=\{2,3\}$, $v=X$. Then $(S,\cap,\subseteq)$ becomes a posemigroup with the following ordering:
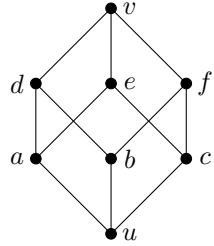
\[
 \begin{tikzpicture}[scale=0.5]
 \draw[fill=black] (0,1) circle (4pt);
 \draw[fill=black] (0,-1) circle (4pt);
 \draw[fill=black] (2,1) circle (4pt);
 \draw[fill=black] (-2,1) circle (4pt);
 \draw[fill=black] (2,-1) circle (4pt);
 \draw[fill=black] (-2,-1) circle (4pt);
 \draw[fill=black] (0,3) circle (4pt);
 \draw[fill=black] (0,-3) circle (4pt);
 \node at (-2.5,-1) {$a$};
 \node at (0.5,-1) {$b$};
 \node at (2.5,-1) {$c$};
 \node at (-2.5,1) {$d$};
 \node at (0.5,1) {$e$};
 \node at (2.5,1) {$f$};
 \node at (0.5,-3) {$u$};
 \node at (0.5,3) {$v$};
 \draw (0,1) -- (2,-1) -- (0,-3) -- (-2,-1) -- (0,1);
 \draw (0,-1) -- (2,1) -- (0,3) -- (-2,1) -- (0,-1);
 \draw (0,1) -- (0,3);
 \draw (0,-1) -- (0,-3);
 \draw (2,1) -- (2,-1);
 \draw (-2,1) -- (-2,-1);
 \end{tikzpicture}
\]
We examine that $
\IDD(S)=\left\{u\da,a\da,b\da,c\da,d\da,e\da,f\da,S\right\}.$ Set
\[
\A=\big\{\{x\}\mid x\in S\big\}\cup\big\{\{b,c\},\{u,b\},\{u,c\}\big\}.
\]
Denoted by $H=\{d,e,f\}\da, I_{1}=\{d,e\}\da, I_{2}=\{d,f\}\da, I_{3}=\{e,f\}\da, U_1= \{a,b,c,d,u\}, U_2= \{a,b,c,e,u\}, U_3= \{a,b,c,f,u\}, J_{1}=d\da, J_{2}=e\da, J_{3}=f\da, L=\{a,b,c,u\}, M_{1}=\{a,b,u\}, M_{2}=\{b,c,u\}, M_{3}=\{a,c,u\}, N_{1}=b\da, N_{2}=a\da, N_{3}=c\da, O=u\da, P=\emptyset $. Accordingly, we draw the diagram of $\IDA(S)$ as
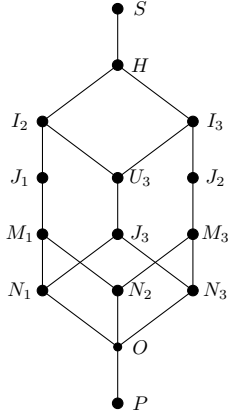
\[
 \begin{tikzpicture}[scale=0.5]
 \draw[fill=black] (0,-1.5) circle (4pt);
 \draw[fill=black] (0,0) circle (3pt);
 \draw[fill=black] (0,1.5) circle (4pt);
 \draw[fill=black] (2,1.5) circle (4pt);
 \draw[fill=black] (-2,1.5) circle (4pt);
 \draw[fill=black] (2,3) circle (4pt);
 \draw[fill=black] (-2,3) circle (4pt);
 \draw[fill=black] (0,3) circle (4pt);
 \draw[fill=black] (2,4.5) circle (4pt);
 \draw[fill=black] (-2,4.5) circle (4pt);
 \draw[fill=black] (0,4.5) circle (4pt);
 \draw[fill=black] (2,6) circle (4pt);
 \draw[fill=black] (-2,6) circle (4pt);
 \draw[fill=black] (0,7.5) circle (4pt);
 \draw[fill=black] (0,9) circle (4pt);
 \tikzstyle{every node}=[scale=0.75]
 \node at (0.6,-1.5) {$P$};
 \node at (0.6,0) {$O$};
 \node at (0.6,1.5) {$N_2$};
 \node at (-2.6,1.5) {$N_1$};
 \node at (2.6,1.5) {$N_3$};
 \node at (2.6,3) {$M_3$};
 \node at (-2.6,3) {$M_1$};
 \node at (-2.6,4.5) {$J_1$};
 \node at (2.6,4.5) {$J_2 $};
 \node at (0.6,4.5) {$U_3$};
 \node at (0.6,3) {$J_3$};
 \node at (-2.6,6) {$I_2$};
 \node at (2.6,6) {$I_3$};
 \node at (0.6,7.5) {$H$};
 \node at (0.6,9) {$S$};
 \draw (0,0) -- (2,1.5) -- (2,3) -- (2,4.5) -- (2,6)-- (0,4.5)-- (-2,6)-- (-2,4.5)-- (-2,3)-- (-2,1.5)-- (0,0);
 \draw (2,6) -- (0,7.5) -- (-2,6);
 \draw (0,0) -- (0,1.5);
 \draw (0,7.5) -- (0,9);
 \draw (2,3) -- (0,1.5) -- (-2,3);
 \draw (0,3)--(0,4.5);
 \draw (0,0)--(0,-1.5);
 \draw (0,3)--(2,1.5);
 \draw (0,3)--(-2,1.5);
 \end{tikzpicture}
\]
and the diagram of $Id_{\mathcal{A}_{0}}(S)$  as
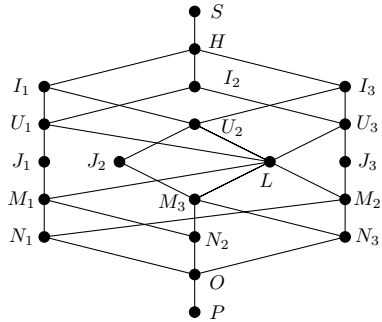
\begin{center}
\begin{tikzpicture}[scale=0.5]
\draw[fill=black] (-4,0) circle (4pt);
 \draw[fill=black] (2,-0) circle (4pt);
 \draw[fill=black] (4,0) circle (4pt);
\draw[fill=black] (0,1) circle (4pt);
 \draw[fill=black] (-2,0) circle (4pt);
 \draw[fill=black] (4,1) circle (4pt);
 \draw[fill=black] (-4,1) circle (4pt);
 \draw[fill=black] (0,2) circle (4pt);
 \draw[fill=black] (4,2) circle (4pt);
 \draw[fill=black] (-4,2) circle (4pt);
 \draw[fill=black] (0,3) circle (4pt);
 \draw[fill=black] (0,4) circle (4pt);
 \draw[fill=black] (0,-1) circle (4pt);
 \draw[fill=black] (4,-1) circle (4pt);
 \draw[fill=black] (-4,-1) circle (4pt);
 \draw[fill=black] (0,-2) circle (4pt);
 \draw[fill=black] (4,-2) circle (4pt);
 \draw[fill=black] (-4,-2) circle (4pt);
 \draw[fill=black] (0,-3) circle (4pt);
 \draw[fill=black] (0,-4) circle (4pt);
 \tikzstyle{every node}=[scale=0.75]
 \node at (1.9,-0.5) {$L$};
 \node at (4.6,0) {$J_3$};
 \node at (-2.6,0) {$J_2$};
 \node at (-4.6,0) {$J_1$};
 \node at (-4.6,1) {$U_{1}$};
 \node at (1,0.9) {$U_{2}$};
 \node at (4.6,1) {$U_{3}$};
 \node at (-4.6,2) {$I_{1}$};
 \node at (1,2.2) {$I_{2}$};
 \node at (4.6,2) {$I_{3}$};
 \node at (0.6,3.2) {$H$};
 \node at (0.6,4) {$S$};
 \node at (-4.6,-1) {$M_{1}$};
 \node at (-0.6,-1.1) {$M_{3}$};
 \node at (4.6,-1) {$M_{2}$};
 \node at (-4.6,-2) {$N_{1}$};
 \node at (0.6,-2.1) {$N_{2}$};
 \node at (4.6,-2) {$N_{3}$};
 \node at (0.6,-3.2) {$O$};
 \node at (0.6,-4) {$P$};
 \draw (4,1) -- (0,2) -- (-4,1);
 \draw (0,1) -- (4,2) -- (0,3) -- (-4,2) -- (0,1);
 \draw (-4,1)--(-4,0)--(-4,-1)--(2,0)--(-4,1);
 \draw (0,1)--(2,-0)--(0,-1)--(2,-0)--(0,1);
 \draw (4,1)--(4,0)--(4,-1)--(2,-0)--(4,1);
 \draw (0,2) -- (0,3);
 \draw (0,3) -- (0,4);
 \draw (4,1) -- (4,2);
 \draw (-4,1) -- (-4,2);
 \draw (0,-1) -- (0,-2) -- (-4,-1);
 \draw (0,-1) -- (4,-2) -- (0,-3) -- (-4,-2) -- (4,-1);
 \draw (0,-2) -- (0,-3);
 \draw (0,-3) -- (0,-4);
 \draw (4,-1) -- (4,-2);
 \draw (-4,-1) -- (-4,-2);
 \draw (0,1)--(-2,0)--(0,-1);
 \end{tikzpicture}
		\end{center}

\end{example}}
	
	\begin{example} \label{example-1}\em
		Let $S$ be a posemigroup with the following multiplication and ordering.
		\vspace{5mm}
		
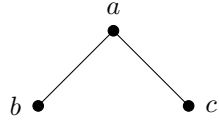
		\begin{minipage}{0.5\linewidth}
			\[
			\begin{tabular}{c|c c c  }
				$\cdot$  &  $a$  &  $b$ & $c$  \\
				\hline
				$a$  & $a$  &  $c$ & $c$ \\
				$b$  & $a$  &  $c$ & $c$ \\
				$c$  & $a$  &  $c$ & $c$ \\
			\end{tabular}
			\]
		\end{minipage}
		\hfill
		\begin{minipage}{.8\linewidth}
			\begin{tikzpicture}
				\draw[fill=black] (0,1) circle (2pt);
				\draw[fill=black] (2,1) circle (2pt);
				\draw[fill=black] (1,2) circle (2pt);
				\node at (2.3,1) {$c$};
				\node at (-0.3,1) {$b$};
				\node at (1,2.3) {$a$};
				\draw (0,1) -- (1,2) -- (2,1);
			\end{tikzpicture}
		\end{minipage}
		\\[0.2cm]

		Since the set $\{b, c\}$ is not  ${\mathcal D}$-admissible we have that it is
		a ${\mathcal D}$-ideal. Moreover, $\{b, c\}$ is closed.
		Hence {{$ \IDD{(S)}=\QQ(S)$}} is
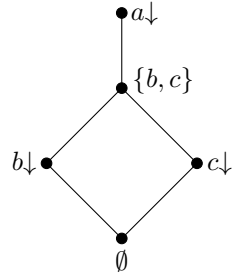
				\[
		\begin{tikzpicture}
			\draw[fill=black] (0,1) circle (2pt);
			\draw[fill=black] (2,1) circle (2pt);
			\draw[fill=black] (1,2) circle (2pt);
			\draw[fill=black] (1,0) circle (2pt);
			\draw[fill=black] (1,3) circle (2pt);
			\node at (1.3,3) {$a\da$};
			\node at (1.55,2.1) {$\{b, c\}$};
			\node at (2.3,1) {$c\da$};
			\node at (-0.3,1) {$b\da$};
			\node at (1,-0.3) {$\emptyset$};
			\draw (0,1) -- (1,2) -- (2,1) -- (1,0) -- (0,1);
			\draw (1,3) -- (1,2);
		\end{tikzpicture}
		\]
		Consider $t: S\to {{\IDD{(S)}} }$ defined as in Theorem \ref{theo-D-5}, and $I=\{b,c\}\subseteq S$. Straightforward checking shows that
		\[
		t(b\vee c)=t(a)=a\da\neq I=\wid{b\da \cup c\da}=t(b)\vee t(c).
		\]
		In addition, for any subset $X$ in $t(S)$, we have $I \neq \bigwedge X$.
		
		Define $f: S \to S$ by $f(s) = a$, $\forall s \in S$.
		Evidently, $f$ is a posemigroup morphism and for every $M \subseteq S$, $f\left(\ov{M\da}\right) \subseteq\ov{f(M)\da} = a\da$.  Therefore, $f$ is closure preserving.  But
		\[
		f\left(\ov{\{b,c\}}\right) =  f\left(\{b,c\}\right) = %
		\{a\} \varsubsetneq a\da = \ov{f\left(\{b,c\}\right)}.
		\]
	\end{example}

	\begin{example}\em \label{example-2}
		Let $S$ be a posemigroup with the following multiplication and ordering.
		\\
		
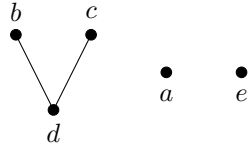
		\begin{minipage}{0.5\linewidth}
			$$\begin{array}{c|ccccc}
				\cdot & a & b & c & d & e\\
				\hline
				a & b & a & a & a & a\\
				b & a & b & b & b & b\\
				c & a & b & b & b & b\\
				d & a & b & b & d & d\\
				e & a & b & c & d & e
			\end{array}$$
		\end{minipage}
		\hfill
		\begin{minipage}{.8\linewidth}
			\begin{tikzpicture}
				\draw[fill=black] (0.5,0) circle (2pt);
				\draw[fill=black] (-0.5,0) circle (2pt);
				\draw[fill=black] (0,-1) circle (2pt);
				\draw[fill=black] (1.5,-0.5) circle (2pt);
				\draw[fill=black] (2.5,-0.5) circle (2pt);
				\node at (-0.5,0.3) {$b$};
				\node at (0.5,0.3) {$c$};
				\node at (0,-1.3) {$d$};
				\node at (1.5,-0.8) {$a$};
				\node at (2.5,-0.8) {$e$};
				\draw (0.5,0) -- (0,-1) -- (-0.5,0) ;
			\end{tikzpicture}
		\end{minipage}
		\\
		\\
		\\
		Then the diagrams of ${\IDD{(S)}}$ and $\QQ(S)$  are
		
		\begin{center}
			\begin{tikzpicture}
				\draw[fill=black] (5,0) circle (2pt);
				\draw[fill=black] (2.5,1) circle (2pt);
				\draw[fill=black] (5,1) circle (2pt);
				\draw[fill=black] (7.5,1) circle (2pt);
				\draw[fill=black] (5/3,2) circle (2pt);
				\draw[fill=black] (10/3,2) circle (2pt);
				\draw[fill=black] (15/3,2) circle (2pt);
				\draw[fill=black] (20/3,2) circle (2pt);
				\draw[fill=black] (25/3,2) circle (2pt);
				\draw[fill=black] (0,3) circle (2pt);
				\draw[fill=black] (2,3) circle (2pt);
				\draw[fill=black] (4,3) circle (2pt);
				\draw[fill=black] (6,3) circle (2pt);
				\draw[fill=black] (8,3) circle (2pt);
				\draw[fill=black] (10,3) circle (2pt);
				\draw[fill=black] (2,4) circle (2pt);
				\draw[fill=black] (4,4) circle (2pt);
				\draw[fill=black] (6,4) circle (2pt);
				\draw[fill=black] (8,4) circle (2pt);
				\draw[fill=black] (5,5) circle (2pt);
				\node at (5,-0.3) {$\emptyset$};
				\node at (2.2,1) {$a\da$};
				\node at (4.7,0.9) {$d\da$};
				\node at (7.8,1) {$e\da$};
				\node at (5/3-0.6,2) {$\{a,d\}$};
				\node at (10/3-0.3,2) {$b\da$};
				\node at (15/3-0.6,2) {$\{a,e\}$};
				\node at (20/3+0.3,2) {$c\da$};
				\node at (25/3+0.5,2) {$\{d,e\}$};
				\node at (-0.7,3) {$\{a,b,d\}$};
				\node at (1.3,3) {$\{a,c,d\}$};
				\node at (3.3,3) {$\{a,d,e\}$};
				\node at (6.7,3) {$\{b,c,d\}$};
				\node at (8.7,3) {$\{b,d,e\}$};
				\node at (10.7,3) {$\{c,d,e\}$};
				\node at (1.2,4.1) {$\{a,b,c,d\}$};
				\node at (3.3,4.1) {$\{a,b,d,e\}$};
				\node at (6.7,4.1) {$\{a,c,d,e\}$};
				\node at (8.8,4.1) {$\{b,c,d,e\}$};
				\node at (5,5.3) {$S$};
				\draw (5,0) -- (2.5,1) -- (5/3,2) -- (0,3) -- (2,4) -- (5,5) -- (8,4) -- (10,3) -- (25/3,2) -- (7.5,1) -- (5,0);
				\draw (2.5,1) -- (15/3,2) -- (4,3) -- (4,4) -- (5,5);
				\draw (5,0) -- (5,1) -- (5/3,2) -- (2,3) -- (2,4);
				\draw (5,1) -- (10/3,2) -- (0,3) -- (4,4);
				\draw (5,1) -- (25/3,2);
				\draw (5,1) -- (20/3,2) -- (2,3) -- (6,4) -- (5,5);
				\draw (7.5,1) -- (15/3,2);
				\draw (7.5,1) -- (25/3,2) -- (4,3) -- (6,4);
				\draw (5/3,2) -- (4,3);
				\draw (10/3,2) -- (6,3) -- (8,4);
				\draw (10/3,2) -- (8,3) -- (4,4);
				\draw (20/3,2) -- (10,3) -- (6,4);
				\draw (20/3,2) -- (6,3) -- (2,4);
				\draw (25/3,2) -- (8,3) -- (8,4);
			\end{tikzpicture}\\
			The diagram of ${\IDD{(S)}}$
		\end{center}

		\begin{center}
			\begin{tikzpicture}
				\draw[fill=black] (5,0) circle (2pt);
				\draw[fill=black] (2.5,2.5) circle (2pt);
				\draw[fill=black] (5,1.25) circle (2pt);
				\draw[fill=black] (7.5,2.5) circle (2pt);
				\draw[fill=black] (3.75,2.5) circle (2pt);
				\draw[fill=black] (5,2.5) circle (2pt);
				\draw[fill=black] (5.625,4.375) circle (2pt);
				\draw[fill=black] (5,3.75) circle (2pt);
				\draw[fill=black] (6.875,3.125) circle (2pt);
				\draw[fill=black] (5,5) circle (2pt);
				\node at (5,-0.3) {$\emptyset$};
				\node at (2.2,2.5) {$a\da$};
				\node at (4.7,1.25) {$d\da$};
				\node at (7.8,2.5) {$e\da$};
				\node at (3.4,2.5) {$b\da$};
				\node at (5.4,2.5) {$c\da$};
				\node at (7.4,3.2) {$\{d,e\}$};
				\node at (5.632,3.6) {$\{b,c,d\}$};
				\node at (6.5,4.4) {$\{b,c,d,e\}$};
				\node at (5,5.3) {$S$};
				\draw (5,0) -- (2.5,2.5) -- (5,5) -- (5.625,4.375) -- (6.875,3.125) -- (7.5,2.5) -- (5,0);
				\draw (5,0) -- (5,1.25) -- (3.75,2.5) -- (5,3.75) -- (5,2.5) -- (5,1.25);
				\draw (5,3.75) -- (5.625,4.375);
				\draw (5,1.25) -- (6.875,3.125);
			\end{tikzpicture}\\
			The diagram of $\QQ(S)$
		\end{center}
		
		In this example, ${\IDD{(S)}}=\PP(S)$.
		Define two mappings $t: S \to {\IDD{(S)}}$ and $f: S \to \QQ(S)$ by $t(s) = s \da$, $f(s) = s \da$, $\forall s \in S$. By Theorem \ref{theo-reflection-D}, there is a unique morphism $g:{\IDD{(S)}} \to \QQ(S)$ given by $g(D)= \bigvee_{\QQ(S)}  f(D)$, for every  {${\mathcal D}$}-ideal $D$ of $S$, such that the following diagram commutes.
		\[
		\bfig\qtriangle/->`->`-->/[S` {\IDD{(S)}}`\QQ(S);t`f`g]\efig
		\]
		Evidently, $f$ is an order-embedding. However, $g(\{a,b,c,d\}) = g(\{a,b,d,e\}) = S$, which indicates that $g$ is not an order-embedding.
	\end{example}

	\begin{example}\em \label{example-3}
		Let $S_1$ be a posemigroup with the following multiplication and ordering.\\
		
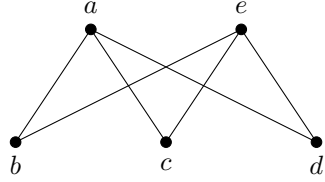
		\begin{minipage}{0.5\linewidth}
			$$\begin{array}{c|ccccc}
				\cdot & a & b & c & d & e\\
				\hline
				a & a & a & a & a & a\\
				b & a & a & a & a & a\\
				c & a & a & a & a & a\\
				d & a & a & a & a & a\\
				e & a & a & a & a & a
			\end{array}$$
		\end{minipage}
		\hfill
		\begin{minipage}{.8\linewidth}
			\begin{tikzpicture}
				\draw[fill=black] (0,0) circle (2pt);
				\draw[fill=black] (2,0) circle (2pt);
				\draw[fill=black] (1,1.5) circle (2pt);
				\draw[fill=black] (4,0) circle (2pt);
				\draw[fill=black] (3,1.5) circle (2pt);
				\node at (0,-0.3) {$b$};
				\node at (2,-0.3) {$c$};
				\node at (1,1.8) {$a$};
				\node at (4,-0.3) {$d$};
				\node at (3,1.8) {$e$};
				\draw (0,0) -- (1,1.5) -- (2,0) -- (3,1.5) -- (4,0);
				\draw (0,0) -- (3,1.5);
				\draw (4,0) -- (1,1.5);
			\end{tikzpicture}
		\end{minipage}\\
		
		Let $S_2$ be a posemigroup with the following multiplication and ordering.\\
		
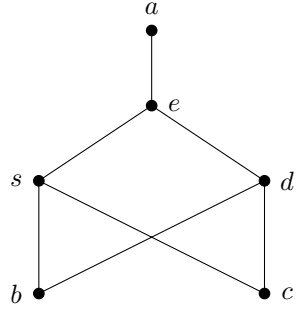
		\begin{minipage}{0.5\linewidth}
			$$\begin{array}{c|cccccc}
				\cdot & s & a & b & c & d & e\\
				\hline
				s & a & a & a & a & a & a\\
				a & a & a & a & a & a & a\\
				b & a & a & a & a & a & a\\
				c & a & a & a & a & a & a\\
				d & a & a & a & a & a & a\\
				e & a & a & a & a & a & a
			\end{array}$$
		\end{minipage}
		\hfill
		\begin{minipage}{.8\linewidth}
			\begin{tikzpicture}
				\draw[fill=black] (0,0) circle (2pt);
				\draw[fill=black] (3,0) circle (2pt);
				\draw[fill=black] (0,1.5) circle (2pt);
				\draw[fill=black] (3,1.5) circle (2pt);
				\draw[fill=black] (1.5,2.5) circle (2pt);
				\draw[fill=black] (1.5,3.5) circle (2pt);
				\node at (-0.3,0) {$b$};
				\node at (3.3,0) {$c$};
				\node at (-0.3,1.5) {$s$};
				\node at (3.3,1.5) {$d$};
				\node at (1.8,2.5) {$e$};
				\node at (1.5,3.8) {$a$};
				\draw (0,0) -- (0,1.5) -- (3,0) -- (3,1.5) -- (0,0);
				\draw (0,1.5) -- (1.5,2.5) -- (3,1.5);
				\draw (1.5,2.5) -- (1.5,3.5);
			\end{tikzpicture}
		\end{minipage}\\
		
		Define $\iota: S_1 \to S_2$ by $\iota(x) = x$, for every $x \in S_1$. Then $\iota$ is  $\D$-admissible preserving but not closure preserving since $\iota\left(\ov{\{b, c\}}\right)=\{b, c, d\} \nsubseteq \{b, c\}=\ov{\iota(\{b, c\})}$.
	\end{example}
	
	
	{\section{Conclusion}

	In this paper, we mainly investigate the category of marked posemigroups and marked
	quantales with the aim of including within a uniﬁed landscape several classes of
	natural ordered structures arising in fuzzy logic. Several results concerning partial frames have been generalized to our framework.
	Our main results establish reflectors to
	quantales from marked quantales. The existence of such reflectors is important,
	interesting, and useful. Among other things, it serves to justify the chosen
	formalism for the marked quantales. We have also discussed another reflection to quantales from
	the category of posemigroups where morphisms are closure and multiplication preserving mappings.
	
	
	The future theoretical work will be about the following issue. It is well known
	that the category $\UQUAN${} of unital quantales is
	the category \textsf{Mon}(\textbf{Sup}) of monoid
	objects in the category \textbf{Sup} of $\bigvee$-semilattices, with respect to the tensor
	product. Since we are defining a larger category $\MQUAN${} of
	marked quantales, it is
	natural to ask whether the category $\UMQUAN{}$ of unital marked quantales is \textsf{Mon}($X$) for some suitably
	defined category $X$ that generalises \textbf{Sup}. The natural guess for such $X$
	would be the category of marked semilattices. We intend to study this problem in more detail.
	One of the first questions that should be answered is whether one needs some other
	conditions besides conditions (A1) and (A2).
	

}
	
	\section*{Acknowledgement}

	Research of the first named author was supported by the Guangdong Basic and Applied Basic Research Foundation, China, No. 2020A1515010206 and No. 2021A1515010248,  the Science and Technology Program of Guangzhou, China, No. 202102080074. Research of the second named author was supported by the Austrian Science Fund (FWF), project I~4579-N, and the Czech Science Foundation (GA\v CR), project 20-09869L, entitled ``The many facets of orthomodularity''.

    We would like to thank Shuang Li and Sergey A. Solovyov for useful comments.
	


\end{document}